\documentclass[12pt,reqno]{amsart}

\usepackage{amssymb}
\usepackage{amsthm}
\usepackage{amsmath}
\usepackage{a4wide}
\usepackage{latexsym}
\usepackage{mathrsfs}
\usepackage[v2,tips]{xy}
\usepackage{float}
\usepackage[T1]{fontenc}
\usepackage[latin1]{inputenc}
\usepackage{rotating}

\newtheorem{theorem}{Theorem}[section]
\newtheorem{lemma}[theorem]{Lemma}
\newtheorem{proposition}[theorem]{Proposition}
\newtheorem{corollary}[theorem]{Corollary}

\theoremstyle{definition}
\newtheorem{definition}[theorem]{Definition}

\newtheorem{remark}[theorem]{Remark}

\numberwithin{equation}{section}

\newcounter{smallromans}

\newenvironment{romanenumerate}
{\begin{list}{{\normalfont\textrm{(\roman{smallromans})}}}%
    {\usecounter{smallromans}\setlength{\itemindent}{0cm}%
      \setlength{\leftmargin}{5.5ex}\setlength{\labelwidth}{5.5ex}%
      \setlength{\topsep}{0.4ex}\setlength{\partopsep}{0.2ex}%
      \setlength{\itemsep}{0.4ex}}}%
  {\end{list}}

\newcommand{\romanref}[1]{{\normalfont\textrm{(\ref{#1})}}}

\newcounter{smallalphs}

\newenvironment{alphenumerate}
{\begin{list}{{\normalfont\textrm{(\alph{smallalphs})}}}%
    {\usecounter{smallalphs}\setlength{\itemindent}{0cm}%
      \setlength{\leftmargin}{5.5ex}\setlength{\labelwidth}{5.5ex}%
      \setlength{\topsep}{0.4ex}\setlength{\partopsep}{0.2ex}%
      \setlength{\itemsep}{0.4ex}}}%
  {\end{list}}

\newcommand{\alphref}[1]{{\normalfont\textrm{(\ref{#1})}}}

\renewcommand{\le}{\ensuremath{\leqslant}}

\renewcommand{\ge}{\ensuremath{\geqslant}}

\newcommand{\N}{\mathbb{N}}
\newcommand{\Z}{\mathbb{Z}}
\newcommand{\Q}{\mathbb{Q}}
\newcommand{\R}{\mathbb{R}}
\newcommand{\C}{\mathbb{C}}
\newcommand{\I}{\mathbb{I}}
\newcommand{\K}{\mathbb{K}}
\newcommand{\subfield}{\mathbb{L}}

\renewcommand{\phi}{\ensuremath{\varphi}}
\renewcommand{\epsilon}{\ensuremath{\varepsilon}}

\newcommand{\locsupp}{\operatorname{locsupp}}
\newcommand{\supp}{\operatorname{supp}}
\newcommand{\spa}{\operatorname{span}}
\newcommand{\clspa}{\overline{\operatorname{span}}\,}
\newcommand{\ran}{\operatorname{ran}}
\newcommand{\rank}{\operatorname{rank}}
\newcommand{\weight}{\operatorname{weight}}
\newcommand{\age}{\operatorname{age}}
\newcommand{\dist}{\operatorname{dist}}
\newcommand{\real}{\operatorname{Re}}

\newcommand{\smashw}[2][l]{{\text{\makebox[0pt][#1]{$#2$}}}}

\begin{document}
\title[Ideal structure  of the algebra of bounded operators]{Ideal structure
  of the algebra of bounded operators acting on a Banach space} 
\dedicatory{In memoriam: Uffe Haagerup (1949--2015)} 
\subjclass[2010]%
{Primary 46H10,
47L10;
Secondary
46B03,
47L20}
\author[T.~Kania]{Tomasz Kania}
\address{Department of Mathematics and Statistics, Fylde
  College, Lancaster University, Lancaster LA1 4YF, United Kingdom\newline
\indent 
{\normalfont{Current address:}} School of Mathematical Sciences,
 Western Gateway Building, University College Cork, Cork, Ireland}
\email{tomasz.marcin.kania@gmail.com}
\author[N.~J.~Laustsen]{Niels Jakob  Laustsen} 
\address{Department of Mathematics and Statistics, Fylde
  College, Lancaster University, Lancaster LA1 4YF, United Kingdom}
\email{n.laustsen@lancaster.ac.uk}
\begin{abstract} 
  We construct a Banach space~$Z$ such that the lattice of closed
  two-sided ideals of the Banach algebra~$\mathscr{B}(Z)$ of bounded
  operators on~$Z$ is as follows:
  \[ \{0\}\subset \mathscr{K}(Z)\subset\mathscr{E}(Z)
     \raisebox{-.5ex}%
              {\ensuremath{\overset{\begin{turn}{30}$\subset$\end{turn}}%
                  {\begin{turn}{-30}$\subset$\end{turn}}}}\!\!%
     \begin{array}{c}\mathscr{M}_1\\[1mm]\mathscr{M}_2\end{array}\!\!\!%
     \raisebox{-1.25ex}%
              {\ensuremath{\overset{\raisebox{1.25ex}{\ensuremath{\begin{turn}{-30}$\subset$\end{turn}}}}%
                  {\raisebox{-.25ex}{\ensuremath{\begin{turn}{30}$\subset$\end{turn}}}}}}\,\mathscr{B}(Z). \]%
   We then determine which kinds of approximate identities
   (bounded/left/right), if any, each of the four non-trivial closed
   ideals of~$\mathscr{B}(Z)$ contains, and we show that the maximal
   ideal~$\mathscr{M}_1$ is generated as a left ideal by two
   operators, but not by a single operator, thus answering a question
   left open in our collaboration with Dales, Kochanek and Koszmider
   (\emph{Studia Math.}~2013). In contrast, the other maximal
   ideal~$\mathscr{M}_2$ is not finitely generated as a left ideal.

   The Banach space~$Z$ is the direct sum of Argyros and Haydon's
   Banach space~$X_{\text{AH}}$ which has very few operators and a
   certain subspace~$Y$ of~$X_{\text{AH}}$. The key property of~$Y$ is
   that every bounded operator from~$Y$ into~$X_{\text{AH}}$ is the sum of a
   scalar multiple of the inclusion map and a compact operator.
\end{abstract}
\maketitle
\section{Introduction and statement of main results}%
\label{section1}
\noindent 
A Banach space~$E$ has \emph{very few operators} if~$E$ is
infinite-dimensional and every bounded operator on~$E$ is the sum of a
scalar multiple of the identity operator and a compact operator; that
is, $\mathscr{B}(E) = \K I_E + \mathscr{K}(E)$, where $\K = \R$ or $\K
= \C$ denotes the scalar field of~$E$. Resolving a famous,
long-standing open problem, Argyros and Haydon~\cite{ah} established
the existence of such Banach spaces by proving the following
spectacular result.

\begin{theorem}[Argyros and Haydon]\label{thmAH} There exists a
  Banach space~$X_{\text{\normalfont{AH}}}$ such that:
  \begin{romanenumerate}
  \item\label{thmAH1} $X_{\text{\normalfont{AH}}}$ has very few
    operators;
  \item\label{thmAH2} $X_{\text{\normalfont{AH}}}$ has a shrinking
    Schauder basis;
  \item\label{thmAH3} the dual space of $X_{\text{\normalfont{AH}}}$
    is isomorphic to~$\ell_1$.
  \end{romanenumerate}
\end{theorem}

The starting point of the present paper is the observation
that~$X_{\text{\normalfont{AH}}}$ contains a subspace~$Y$ which has
certain special properties, as specified in following theorem; of
these, property~\romanref{prop3ofY} is by far the most important, and
also the hardest to achieve.  We are deeply grate\-ful to Professor
Argyros for having explained to us how to contruct such a subspace;
details of its construction will be given in Section~\ref{section2}.
\begin{theorem}\label{thesubspaceY}
  Argyros and Haydon's Banach space~$X_{\text{\normalfont{AH}}}$
  contains a closed, in\-finite-dimen\-sional subspace~$Y$ which has
  the following four properties:
  \begin{romanenumerate}
  \item\label{prop4ofY} $Y$ is the closed linear span of a certain
    subsequence of the Schauder basis
    for~$X_{\text{\normalfont{AH}}},$ and hence~$Y$ has a shrinking
    Schauder basis;
  \item\label{prop1ofY} $Y$ has infinite co\-dimen\-sion
    in~$X_{\text{\normalfont{AH}}},$ and hence $Y$ is 
    un\-com\-ple\-ment\-ed in~$X_{\text{\normalfont{AH}}};$
  \item\label{prop2ofY} the dual space of~$Y$
    is isomorphic to~$\ell_1;$
  \item\label{prop3ofY} every bounded operator from~$Y$
    into~$X_{\text{\normalfont{AH}}}$ is the sum of a scalar multiple
    of the inclusion map $J\colon Y\to X_{\text{\normalfont{AH}}}$ and
    a compact operator.
  \end{romanenumerate}
\end{theorem}

In the remainder of this paper, we shall consider the Banach space
\begin{equation}\label{defnspaceZ} 
  Z=X_{\text{\normalfont{AH}}}\oplus Y, 
\end{equation}
where~$X_{\text{\normalfont{AH}}}$ and~$Y$ are as in
Theorems~\ref{thmAH} and~\ref{thesubspaceY}, respectively. For
definiteness, we shall equip~$Z$ with the $\ell_\infty$-norm; that is,
$\|(x,y)\| = \max\{\|x\|,\|y\|\}$ for $x\in
X_{\text{\normalfont{AH}}}$ and $y\in Y$; all our results will,
however, be of an isomorphic nature, so that any equivalent norm will
do.  Theorems~\ref{thmAH}\romanref{thmAH1}
and~\ref{thesubspaceY}\romanref{prop1ofY}+\romanref{prop3ofY} imply
that every bounded operator~$T$ on~$Z$ has a unique representation as
an operator-valued $(2\times 2)$-matrix of the form
\begin{equation}\label{matrixrepofoperator}
T = \begin{pmatrix} \alpha_{1,1}I_{X_{\text{\normalfont{AH}}}} +
  K_{1,1} & \alpha_{1,2} J + K_{1,2}\\ K_{2,1} & \alpha_{2,2}I_Y + 
  K_{2,2} \end{pmatrix},
\end{equation}
where $\alpha_{1,1}, \alpha_{1,2}$ and~$\alpha_{2,2}$ are scalars,
$I_{X_{\text{\normalfont{AH}}}}$ and~$I_Y$ denote the identity
operators on~$X_{\text{\normalfont{AH}}}$ and $Y$, respectively,
$J\colon Y\to X_{\text{\normalfont{AH}}}$ is the inclusion map, and
the operators \mbox{$K_{1,1}\colon X_{\text{\normalfont{AH}}}\to
X_{\text{\normalfont{AH}}}$}, $K_{1,2}\colon Y\to
X_{\text{\normalfont{AH}}}$, \mbox{$K_{2,1}\colon
  X_{\text{\normalfont{AH}}}\to Y$} and $K_{2,2}\colon Y\to Y$ are
compact.

Using this notation, we see that the sets
\begin{equation}\label{defnmaxideal}
  \mathscr{M}_1 = \{T\in\mathscr{B}(Z) : \alpha_{2,2} =
  0\}\qquad\text{and}\qquad \mathscr{M}_2 = \{T\in\mathscr{B}(Z) :
  \alpha_{1,1} = 0\}
\end{equation}
are maximal two-sided ideals of codimension one
in~$\mathscr{B}(Z)$. Our first main result gives a complete
description of the lattice of closed two-sided ideals
of~$\mathscr{B}(Z)$. Its statement involves the following notion,
which goes back to Kleinecke~\cite{kl}.
\begin{definition}
  A bounded operator on a Banach space~$E$ is \emph{inessential} if it
  belongs to the pre-image of the Jacobson radical of the Calkin
  algebra~$\mathscr{B}(E)\big/\overline{\mathscr{F}(E)}$, where
  $\overline{\mathscr{F}(E)}$ denotes the norm-closure of the ideal of
  finite-rank operators on~$E$. 
\end{definition}

We write~$\mathscr{E}(E)$ for the set of inessential operators on the
Banach space~$E$. This is a closed two-sided ideal of~$\mathscr{B}(E)$
which is proper if (and only if) $E$ is infinite-dimensional.

\begin{theorem}\label{2sidedideallatticeofBZ} The Banach
  algebra~$\mathscr{B}(Z)$ of bounded operators on the Banach
  space~$Z$ defined by~\eqref{defnspaceZ} contains exactly six closed
  two-sided ideals, namely
  \[ \spreaddiagramrows{1ex}\spreaddiagramcolumns{-6ex}%
  
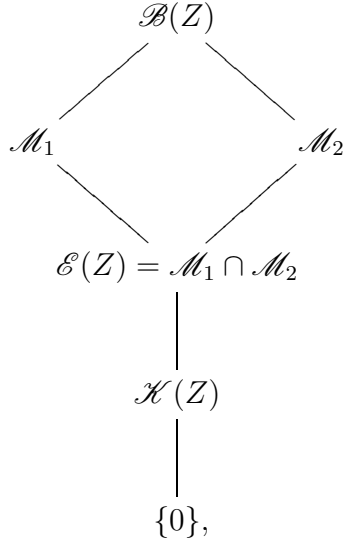
\xymatrix{&
    \mathscr{B}(Z)\ar@{-}[dl]\ar@{-}[dr]\\ \mathscr{M}_1\ar@{-}[dr] &
    & \mathscr{M}_2\ar@{-}[dl]\\ & \mathscr{E}(Z) =
    \mathscr{M}_1\cap\mathscr{M}_2\ar@{-}[d]\\ &
    \mathscr{K}(Z)\ar@{-}[d]\\ & \{0\}\smashw{,}} \]
  where~$\mathscr{M}_1$ and~$\mathscr{M}_1$ are given
  by~\eqref{defnmaxideal}, and the lines denote proper inclusions,
  with the larger ideal at the top.
\end{theorem}
\noindent
We note that in the diagram, above, the smaller ideal has codimension
one in the larger ideal in each of the inclusions, except the
bottom\-most.

\begin{remark}
  Not many infinite-dimensional Banach spaces~$E$ are known for which
  a complete classification of the closed two-sided ideals of~$\mathscr{B}(E)$
  exists. Indeed, to the best of our knowledge at the time of writing,
  the following list contains all such examples:
  \begin{romanenumerate}
  \item\label{idealclass1} the classical sequence spaces $E = \ell_p(\mathbb{I})$ for
    $1\le p <\infty$ and $E = c_0(\mathbb{I})$, where~$\mathbb{I}$ is
    an arbitrary infinite index set; these results are due to
    Calkin~\cite{calkin} for countable~$\mathbb{I}$ and $p=2$;
    Gohberg--Markus--Feldman~\cite{gmf} for countable~$\mathbb{I}$ and
    general~$p$ (including~$c_0$); Gramsch~\cite{gr} and
    Luft~\cite{luft} for $p=2$ and arbitrary~$\mathbb{I}$; and
    Daws~\cite{daws} in full generality;
  \item\label{idealclass2} the $c_0$-direct sum of the sequence of finite-dimensional
    Hilbert spaces of increasing dimension, that is, $E =
    \bigl(\bigoplus_{n\in\N}\ell_2^n\bigr)_{c_0}$, and its dual space
    $\bigl(\bigoplus_{n\in\N}\ell_2^n\bigr)_{\ell_1}$ (see~\cite{llr}
    and~\cite{lsz}, respectively);
  \item\label{idealclass3} $E = X_{\text{\normalfont{AH}}}$ by Theorem~\ref{thmAH},
    above;
  \item\label{idealclass4} Tarbard's variants of the
    Argyros--Haydon space: for each~$n\in\N$, there is a Banach
    space~$E$ such that~$E$ admits a strictly singular operator~$S$
    which is nilpotent of order~$n+1$, and every bounded operator
    on~$E$ has the form $\sum_{j=0}^n\alpha_jS^j +K$ for some scalars
    $\alpha_0,\ldots,\alpha_n$ and a compact operator~$K$
    (see~\cite[Theorem~2.1]{tarb});
  \item\label{idealclass5} $E = C(\Omega)$, where $\Omega$ is the Mr\'{o}wka
    space constructed by Koszmider~\cite{kosz}, assuming the Continuum
    Hypothesis (see~\cite[Theorem~5.5]{kk}; this result has also been
    obtained independently by Brooker (unpublished));
  \item\label{idealclass6} certain Banach spaces constructed by
    Motakis, Puglisi and Zisimopoulou~\cite{mpz}: for every countably
    infinite compact metric space~$\Omega$, there is a Banach space~$E$
    such that the Banach algebra~$\mathscr{B}(E)/\mathscr{K}(E)$ is
    isomorphic to the algebra~$C(\Omega)$ of scalar-valued, continuous
    functions defined on~$\Omega$. (The classification of the closed
    two-sided ideals of~$\mathscr{B}(E)$ is not stated explicitly
    in~\cite{mpz}, but it is an easy consequence of
    \cite[Theorem~5.1]{mpz}, together with the following two facts:
    (1)~$E$~is a $\mathscr{L}_\infty$-space, so it has the bounded
    approximation property, and therefore~$\mathscr{K}(E)$ is the
    minimum non-zero closed two-sided ideal of~$\mathscr{B}(E)$;
    (2)~the closed ideals of the Banach algebra~$C(\Omega)$ for a compact
    Hausdorff space~$\Omega$ are precisely the zero sets of the closed
    subsets of~$\Omega$.)
  \end{romanenumerate}
  In each of the cases~\romanref{idealclass1}--\romanref{idealclass5},
  above, the lattice of closed two-sided ideals of~$\mathscr{B}(E)$ is
  linearly ordered, whereas in case~\romanref{idealclass6}, it is
  infinite. Hence the Banach space~$Z$ given by~\eqref{defnspaceZ}
  appears to be the first Banach space~$E$ for which we have a
  complete classi\-fi\-ca\-tion of the lattice of closed two-sided
  ideals of the Banach alge\-bra~$\mathscr{B}(E)$, and this lattice is
  finite, but it is not linearly ordered.
\end{remark}

\noindent 
\textsl{Note added in proof.} We shall here describe another family of
Banach spaces $E$ such that the lattice of closed two-sided ideals
of~$\mathscr{B}(E)$ is finite and not linearly ordered.  For each
$n\in\N$, we apply \cite[Theorem~10.4]{ah} to obtain Banach spaces
$X_1,\ldots, X_n$, each having very few operators, each having a
Schauder basis, and such that every bounded operator from $X_j$ to
$X_k$ is compact whenever $j,k\in\{1,\ldots,n\}$ are distinct.  Take
$m_1,\ldots,m_n\in\N$, and set $E = X_1^{m_1}\oplus\cdots \oplus
X_n^{m_n}$. Then $\mathscr{K}(E)$ is the smallest non-zero closed
two-sided ideal of~$\mathscr{B}(E)$, and we have
\[ \mathscr{B}(E)/\mathscr{K}(E)\cong M_{m_1}(\K)\oplus\cdots\oplus 
M_{m_n}(\K), \] where $M_m(\K)$ denotes the algebra of scalar-valued
$(m\times m)$-matrices. By Wedderburn's structure theorem (see,
\emph{e.g.}, \cite[Theorem~1.5.9]{da}), this shows that every
finite-dimensional, semi-simple complex algebra can arise as the
Calkin algebra of a Banach space. Moreover, we note that the choice
$m_1=\cdots= m_n = 1$ gives a counterpart of the result of Motakis,
Puglisi and Zisimopoulou that we described in~\romanref{idealclass6},
above, in the case where the underlying space~$\Omega$ is finite.

Returning to the general case where $m_1,\ldots,m_n\in\N$ are
arbitrary, we may consider each bounded operator~$T$ on~$E$ as an
operator-valued $(n\times n)$-matrix $(T_{j,k})_{j,k=1}^n$, where we
have \mbox{$T_{j,k}\in\mathscr{B}(X_k^{m_k},X_j^{m_j})$} for each
$j,k$. Since $M_m(\K)$ is simple for each $m\in\N$, the map
\[ N\mapsto \bigl\{ (T_{j,k})_{j,k=1}^n : T_{j,j}\in\mathscr{K}(X_j^{m_j})\ 
(j\notin N)\bigr\} \] is an order isomorphism of the power set of
$\{1,\ldots,n\}$ onto the lattice of non-zero closed two-sided ideals
of~$\mathscr{B}(E)$. Hence the lattice of closed two-sided ideals
of~$\mathscr{B}(E)$ has~$2^n+1$ elements, and it is not linearly
ordered for $n\ge 2$.

Let us finally remark that the ideal lattices obtained in this way are
different from that of~$\mathscr{B}(Z)$ that we described in
Theorem~\ref{2sidedideallatticeofBZ}, above; for instance, none of
these lattices has precisely six elements.\bigskip

After seeing Argyros and Haydon's main results as they were stated in
Theorem~\ref{thmAH}, above, Dales observed that they imply that the
Banach algebra~$\mathscr{B}(X_{\text{\normalfont{AH}}})$ is amenable
\cite[Prop\-o\-si\-tion~10.6]{ah}, thus disproving a long-standing
conjecture of B.~E.~Johnson. In con\-trast, we note
that~$\mathscr{B}(Z)$ does not share this property.

\begin{proposition}\label{BZnonamenable}
  The Banach algebra~$\mathscr{B}(Z)$ is not amenable.
\end{proposition}

The study of amenability is intimately related to the existence of
approximate identities, as explained in \cite[Section~2.9]{da}, for
instance. Our second main result, which will be proved in
Section~\ref{sectionAIs}, describes what kinds of approximate
identities, if any, can be found in each of the four non-trivial
closed two-sided ideals of~$\mathscr{B}(Z)$. Before we state this
result formally, let us introduce the relevant terminology.

\begin{definition} A net $(e_j)_{j\in\mathbb{J}}$ in a Banach
  algebra~$\mathscr{A}$ is a \emph{left approximate identity} if the
  net $(e_j a)_{j\in\mathbb{J}}$ converges to~$a$ for each
  $a\in\mathscr{A}$, and a \emph{right approximate identity} if the
  net $(ae_j)_{j\in\mathbb{J}}$ converges to~$a$ for each
  $a\in\mathscr{A}$.  If in addition
  $\sup_{j\in\mathbb{J}}\|e_j\|<\infty$, then
  $(e_j)_{j\in\mathbb{J}}$ is a \emph{bounded} left or right
  approximate identity.  A \emph{bounded two-sided approximate
    identity} is a net which is simultaneously a bounded left and
  right approximate identity.
\end{definition}
 
\begin{theorem}\label{thmAIs} \begin{romanenumerate}
  \item\label{thmAIs1} The ideal $\mathscr{M}_1$ has a bounded left
    approximate identity, but it has no right approximate identity.
  \item\label{thmAIs2} The ideal $\mathscr{M}_2$ has a bounded right
    approximate identity, but it has no left approximate identity.
  \item\label{thmAIs3} The ideal $\mathscr{E}(Z) =
    \mathscr{M}_1\cap\mathscr{M}_2$ has no left or right approximate
    identity.
  \item\label{thmAIs4} The ideal $\mathscr{K}(Z)$ has a bounded
    two-sided approximate identity.
  \end{romanenumerate}
\end{theorem}

Our third and final main result uses the Banach space~$Z$ to answer
two questions that were left open in~\cite{dkkkl} regarding the
maximal left ideals of the Banach algebra~$\mathscr{B}(E)$ for an
infinite-dimensional Banach space~$E$. To set the stage for this
result, we require some background information from~\cite{dkkkl},
beginning with the easy observation that, for each non-zero
element~$x$ of~$E$, the set
\begin{equation}\label{fixedideal}
  \mathscr{M\!L}_x = \{ T\in\mathscr{B}(E) : Tx = 0\}
\end{equation}
is a maximal left ideal of~$\mathscr{B}(E)$, and it is generated as a
left ideal by a single operator, namely any projection
$P\in\mathscr{B}(E)$ with $\ker P = \K x$.  The maximal left ideals of
the form~\eqref{fixedideal} were termed \emph{fixed} in~\cite{dkkkl},
inspired by the analogous terminology for ultrafilters, and the
following question was studied extensively:
\begin{quotation}
  \emph{Is every finitely-generated, maximal left ideal of the Banach
    algebra~$\mathscr{B}(E)$ necessarily fixed?}
\end{quotation}
Indeed, a positive answer to this question was established for many
Banach spaces~$E$, but, somewhat surprisingly, it was also shown that
the answer is not always positive: for $E =
X_{\text{\normalfont{AH}}}\oplus\ell_\infty$, the Banach
algebra~$\mathscr{B}(E)$ contains a non-fixed, singly-generated,
maximal left ideal of codimension one, namely
\begin{equation*}
  \mathscr{K}_1 = \left\{\begin{pmatrix} T_{1,1} & T_{1,2}\\ T_{2,1} &
  T_{2,2} \end{pmatrix}\in\mathscr{B}(E) :
  T_{1,1}\ {\normalfont{\text{is compact}}}\right\}.
\end{equation*}  
The Banach space~$E = X_{\text{\normalfont{AH}}}\oplus\ell_\infty$ is
evidently not separable.  We shall show in
Section~\ref{sectionMaxLeftIdeals} that a similar example exists based
on the separable Banach space~$Z$ given by~\eqref{defnspaceZ}. This
example will also enable us to answer another question implicitly left
open in~\cite{dkkkl} (see \cite[Proposi\-tion~2.2]{dkkkl} and the
remark following it) because the non-fixed, finitely-generated maximal
left ideal of $\mathscr{B}(Z)$ that we identify is not generated by
one, but by two operators. More precisely, our result is as follows.

\begin{theorem}\label{Thmnonfixedfgmaxleftideal} 
  The ideals $\mathscr{M}_1$ and $\mathscr{M}_2$ are the only
  non-fixed, maximal left ideals of~$\mathscr{B}(Z),$ and
  \begin{romanenumerate}
  \item\label{Thmnonfixedfgmaxleftideal1} $\mathscr{M}_1$ is generated
    as a left ideal by the two operators
    \begin{equation}\label{ThmnonfixedfgmaxleftidealEq2} 
      \begin{pmatrix} I_{X_{\text{\normalfont{AH}}}} & 0\\ 0 &
        0 \end{pmatrix}\qquad\text{and}\qquad \begin{pmatrix} 0 & J\\ 0
        & 0 \end{pmatrix}, \end{equation} but $\mathscr{M}_1$ is not
    generated as a left ideal by a single bounded operator on~$Z;$
  \item\label{Thmnonfixedfgmaxleftideal2} $\mathscr{M}_2$ is not
    finitely generated as a left ideal.
  \end{romanenumerate}
\end{theorem}

\begin{remark}
In a post on \textsl{Mathematics Stack Exchange}, Petry~\cite{petry}
asked whether there is a one-sided version of the Nakayama lemma, in
the following specific sense: let $R$ be a unital non-commutative
ring, and let $L$ be a finitely-generated left ideal of~$R$ such that
$L = L\cdot L$ (that is, each element of~$L$ can be written as the sum
of products of elements of~$L$). Must $L$ be generated (as a left
ideal) by a single idempotent element?

In reply, Schwiebert outlined an example which shows that the answer
is in general nega\-tive. We observe that our results provide another
such example. Indeed, let \mbox{$R = \mathscr{B}(Z)$}, and let $L =
\mathscr{M}_1$.
Theorem~\ref{Thmnonfixedfgmaxleftideal}\romanref{Thmnonfixedfgmaxleftideal1}
shows that $L$ is finitely generated, but not by a single element
(idempotent or not), while Theorem~\ref{thmAIs}\romanref{thmAIs1} in
tandem with Cohen's Factorization Theorem (see, \emph{e.g.,}
\cite[Corollary~2.9.25]{da}) implies that each element of $L$ can be
written as the product of two elements of~$L$.
Being a Banach algebra, this example has a very different flavour from
Schwiebert's, which is based on an algebra over a finite field
constructed by Andruszkiewicz and Puczy\-\l{}ow\-ski~\cite{ap}.
\end{remark}

\section{The construction of the subspace~$Y$ and the proof of
  Theorem~\ref{thesubspaceY}}\label{section2}
\subsection*{Schauder decompositions} 
Let $E$ be a Banach space. A sequence $(F_j)_{j\in\N}$ of non-zero
subspaces of~$E$ is a \emph{Schauder decomposition} for~$E$ if, for
each $x\in E$, there is a unique sequence $(x_j)_{j\in\N}$, where
$x_j\in F_j$ for each $j\in\N$, such that the series
$\sum_{j=1}^\infty x_j$ is norm-convergent with sum~$x$.  In this
case, for each $n\in\N$, we can define a
projection~$P_n\in\mathscr{B}(E)$ by $P_nx = \sum_{j=1}^n x_j$; this
is the \emph{$n^{\text{th}}$ canonical projection} associated with the
decomposition. The number $\sup_{n\in\N}\|P_n\|$ turns out to be
finite; this is the \emph{decomposition constant}.

A Schauder decomposition $(F_j)_{j\in\N}$ for $E$ is:
\begin{itemize}
\item \emph{shrinking} if $\|x^* - P^*_n x^*\|\to 0$ for each $x^*\in
  E^*;$
\item \emph{finite-dimensional} (or an \emph{FDD} for short) if $\dim
  F_j<\infty$ for each $j\in\N$.\\
  (Note: the case where each~$F_j$ is one-dimensional, say $F_j = \K
  b_j\ (j\in\N)$, corresponds to $(b_j)_{j\in\N}$ being a Schauder
  basis for~$E$.)
\end{itemize}

We shall require the following elementary observation concerning
compact operators into or out of a Banach space with an FDD. It goes
back to at least~\cite[Remark, p.~14]{bp} in the case of a single
Banach space with a Schauder basis. For completeness, we outline a
proof.
\begin{lemma}\label{FDDprojapprox}
  Let $D$ and $E$ be Banach spaces, where~$E$ has an FDD, and denote
  by $(P_n)_{n\in\N}$ the canonical projections associated with this
  FDD.
\begin{romanenumerate}
\item\label{FDDprojapprox1} For each bounded operator $S\colon D\to
  E,$ the following two conditions are equivalent:
  \begin{alphenumerate}
    \item\label{FDDprojapprox1a} $S$ is compact;
    \item\label{FDDprojapprox1b} $\| S - P_nS\|\to 0$ as $n\to\infty$.
  \end{alphenumerate}
\item\label{FDDprojapprox2} Suppose that the FDD for~$E$ is shrinking.
  Then, for each bounded operator \mbox{$T\colon E\to D,$} the following three
  conditions are equivalent:
  \begin{alphenumerate}
    \item\label{FDDprojapprox2a} $T$ is compact;
    \item\label{FDDprojapprox2b} $\| T - TP_n\|\to 0$ as $n\to\infty;$
    \item\label{FDDprojapprox2c} $\| Tx_j\|\to 0$ as $j\to\infty$ for
      every bounded block sequence $(x_j)_{j\in\N}$ with respect to
      the FDD for~$E$.
  \end{alphenumerate}
\end{romanenumerate}
\end{lemma}

\begin{proof} Let $C = \sup_{n\in\N}\|P_n\|<\infty$ be the decomposition constant. 

\romanref{FDDprojapprox1}. The implication
\alphref{FDDprojapprox1b}$\Rightarrow$\alphref{FDDprojapprox1a} is
clear because $P_n$ has finite-dimensional image for each $n\in\N$.
Conversely, suppose contrapositively that, for some $\epsilon>0$ and
each $m\in\N$, there is an integer $n\ge m$ such that $\|
(I_E-P_n)S\|>\epsilon$. By recursion, we can choose a sequence
$(x_j)_{j\in\N}$ of unit vectors in~$D$ and a strictly increasing
sequence $(k_j)_{j\in\N}$ of natural numbers such that
$\| (I_E - P_{k_j})Sx_j\| > \epsilon$ and
$\| (I_E - P_m)Sx_j\| < \epsilon/2$ whenever $j,m\in\N$ and $m\ge
k_{j+1}$.
This implies that
\[ (C+1)\| Sx_{i+j} - Sx_i\|\ge \|(I_E-P_{k_{i+j}})Sx_{i+j}\| - 
\|(I_E-P_{k_{i+j}})Sx_i\| > \frac{\epsilon}{2}\qquad (i,j\in\N), \]
which shows that no subsequence of $(Sx_i)_{i\in\N}$ is Cauchy, and
therefore the operator~$S$ is not compact.

\romanref{FDDprojapprox2}. The equivalence of
conditions~\alphref{FDDprojapprox2a} and~\alphref{FDDprojapprox2b}
follows by dualizing~\romanref{FDDprojapprox1} and using Schauder's
theorem together with the fact that $(P_n^*)_{n\in\N}$ are the
canonical projections associated with an FDD for the dual space~$E^*$
of~$E$.

The implication
\alphref{FDDprojapprox2b}$\Rightarrow$\alphref{FDDprojapprox2c} is
easy because, for every block sequence $(x_j)_{j\in\N}$ in~$E$ and
each $n\in\N$, we can find $j_0\in\N$ such that $P_nx_j = 0$ whenever
$j\ge j_0$.

\alphref{FDDprojapprox2c}$\Rightarrow$\alphref{FDDprojapprox2b}. Suppose
contrapositively that, for some $\epsilon>0$ and each $m\in\N$, there
is an integer $k\ge m$ such that $\| T(I_E-P_k)\| >\epsilon$.  Then we
can find a unit vector $w\in E$ and a further integer $j>k$ such that
$\| T(P_j-P_k)w\| >\epsilon$, and hence we can recursively choose
integers $1\le k_1<j_1\le k_2<j_2\le\cdots$ and unit vectors
$w_1,w_2,\ldots\in E$ such that $\| T(P_{j_i}-P_{k_i})w_i\| >\epsilon$
for each $i\in\N$. This implies that $(x_i)_{i\in\N} :=
((P_{j_i}-P_{k_i})w_i)_{i\in\N}$ is a $2C$-bounded block sequence for
which~\alphref{FDDprojapprox2c} fails.
\end{proof}

\subsection*{The Bourgain--Delbaen
  construction} Argyros and Haydon used the \emph{Bourgain--Del\-baen
  construction}~\cite{bg} to define their Banach
space~$X_{\text{\normalfont{AH}}}$. We shall now summarize those parts
of this method that are required for our present purposes. We follow
the notation and terminology used in~\cite{ah} as far as possible,
with the notable exception that our focus is on both real and
complex scalars, whereas \cite{ah} considered real scalars only. For
this reason, it is convenient to introduce a single symbol for the
following countable, dense subfield of the scalar field~$\K$ that will
play the role of the rationals in the real case:
\begin{equation}\label{defnsubfieldL}
\subfield = \begin{cases} \Q\ 
&\text{for}\ \K=\R\\ \Q+\mathrm{i}\Q\ &\text{for}\ \K=\C. \end{cases} 
\end{equation}

For a (non-empty, countable) set~$\Gamma$, we consider the Banach
spaces
\[ \ell_\infty(\Gamma) = \Bigl\{ x\colon\Gamma\to\K :
\sup_{\gamma\in\Gamma}|x(\gamma)|<\infty\Bigr\}\quad\text{and}\quad
\ell_1(\Gamma) = \biggl\{ x^*\colon\Gamma\to\K :
\sum_{\gamma\in\Gamma}|x^*(\gamma)|<\infty\biggr\}, \]
and identify~$\ell_\infty(\Gamma)$ with the dual space
of~$\ell_1(\Gamma)$ via the duality bracket
\[ \langle x^*, x\rangle = \sum_{\gamma\in\Gamma} x^*(\gamma)
x(\gamma)\qquad (x^*\in\ell_1(\Gamma),\, x\in\ell_\infty(\Gamma)). \]
We write $e_\gamma$ and~$e^*_\gamma$ for the elements
of~$\ell_\infty(\Gamma)$ and~$\ell_1(\Gamma)$, respectively, given
by \[ e_\gamma(\gamma) = 1 = e^*_\gamma(\gamma)\qquad\text{and}\qquad
e_\gamma(\eta) = 0 = e^*_\gamma(\eta)\qquad
(\eta\in\Gamma\setminus\{\gamma\}). \] Let $p=1$ or $p=\infty$. Then
$\supp x$ denotes the \emph{support} of an element
$x\in\ell_p(\Gamma)$. Given a non-empty subset~$\Delta$ of~$\Gamma$,
we identify $\ell_p(\Delta)$ with the subspace $\{x\in\ell_p(\Gamma) :
\supp x\subseteq\Delta\}$ of~$\ell_p(\Gamma)$.

The Bourgain--Delbaen construction, as Argyros and Haydon present it,
begins with the singleton set~$\Delta_1 = \{1\}$ and the functional
$c_1^* =0$.  A sequence $(\Delta_n)_{n\in\N}$ of non-empty, finite,
disjoint sets is then defined recursively, together with functionals
$c_\gamma^*\in\spa\{ e_\eta^* : \eta\in\Gamma_n\}$ for each $n\in\N$
and $\gamma\in\Delta_{n+1}$, where $\Gamma_n :=
\bigcup_{j=1}^n\Delta_j$, in such a way that the sequence
\[ (d_\gamma^*)_{\gamma\in\Gamma} := (e^*_\gamma -
c^*_\gamma)_{\gamma\in\Gamma} \] is a Schauder basis for the Banach
space~$\ell_1(\Gamma)$, where $\Gamma := \bigcup_{j\in\N}\Delta_j$,
endowed with the lexicographic order induced by
$\Delta_1,\Delta_2,\ldots$ (The finite sets $\Delta_1,\Delta_2,\ldots$
are \emph{a priori} unordered; they can each be given an arbitrary
linear order to ensure that~$\Gamma$ is linearly ordered.)  In
particular, the finite-dimensional subspaces
$\spa\{d_\gamma^*:\gamma\in\Delta_n\}\ (n\in\N)$ form an FDD
for~$\ell_1(\Gamma)$. We write $P_{(0,\,n]}^*$ for the $n^{\text{th}}$
canonical projection on~$\ell_1(\Gamma)$ associated with this
decomposition;  that is, $P_{(0,\,n]}^*$ is given by
$P_{(0,\,n]}^*d_\gamma^* = d_\gamma^*$ if $\gamma\in\Gamma_n$ and
$P_{(0,\,n]}^*d_\gamma^* = 0$ otherwise. For later reference, we note
that the image of~$P^*_{(0,\,n]}$ is given by 
\begin{equation}\label{imagepfPnstar}
  \spa\{d_\gamma^* : \gamma\in\Gamma_n\} = \spa\{e_\gamma^*:
  \gamma\in\Gamma_n\} = \ell_1(\Gamma_n). \end{equation}

Let $(d_\gamma)_{\gamma\in\Gamma}$ be the sequence of coordinate
functionals in \mbox{$\ell_1(\Gamma)^* = \ell_\infty(\Gamma)$}
associated with the Schauder basis $(d_\gamma^*)_{\gamma\in\Gamma}$
for~$\ell_1(\Gamma)$. The \emph{Bourgain--Delbaen space}~$X(\Gamma)$
determined by the set~$\Gamma$ is now defined as the closed subspace
of~$\ell_\infty(\Gamma)$ spanned by \mbox{$\{d_\gamma :
  \gamma\in\Gamma\}$}, so that, by definition,
$(d_\gamma)_{\gamma\in\Gamma}$ is a Schauder basis
for~$X(\Gamma)$. Denote by~$P_{(0,\,n]}$ the adjoint of the
projection~$P_{(0,\,n]}^*$ for each $n\in\N$. Since the image
of~$P_{(0,\,n]}$ is equal to $\spa\{d_\gamma : \gamma\in\Gamma_n\}$, we 
may consider~$P_{(0,\,n]}$ as an operator into~$X(\Gamma)$. We observe 
that the subspaces 
\[ M_n := \spa\{d_\gamma:\gamma\in\Delta_n\}\qquad  (n\in\N) \] 
form an FDD for~$X(\Gamma)$, and $(P_{(0,\,n]}|X(\Gamma))_{n\in\N}$ 
are the associated projections.

Let $n\in\N$. By~\eqref{imagepfPnstar}, we may regard~$P^*_{(0,\,n]}$ as a 
surjection onto~$\ell_1(\Gamma_n)$.  The adjoint of this operator, which 
we shall denote
by~$i_n\colon\ell_\infty(\Gamma_n)\to\ell_\infty(\Gamma)$, plays an
important role in the study of Bourgain--Delbaen spaces.  It is an
extension operator, in the sense that $i_n(x)(\gamma) = (x)(\gamma)$
for each $x\in\ell_\infty(\Gamma_n)$ and $\gamma\in\Gamma_n$, and it
satisfies
\begin{equation}\label{theembeddingin}
    \|x\|_\infty\le
    \|i_n(x)\|_\infty\le M\|x\|_\infty\qquad
    (x\in\ell_\infty(\Gamma_n)), 
\end{equation}
where~$M$ is the basis constant
of~$(d_\gamma^*)_{\gamma\in\Gamma}$. We can describe~$i_n$ explicitly
by the formula $i_n = P_{(0,\,n]}|_{\ell_\infty(\Gamma_n)}$. In
  particular, its image is spanned by $\{d_\gamma :
  \gamma\in\Gamma_n\}$, and so we may regard~$i_n$ as an operator
  from~$\ell_\infty(\Gamma_n)$ into~$X(\Gamma)$.

Let $x\in\spa\{d_\gamma : \gamma\in\Gamma\}$. By the \emph{range}
of~$x$, we understand the smallest interval~$\I$ of~$\N$ such that
$x\in\spa\bigl\{d_\gamma :
\gamma\in\bigcup_{i\in\I}\Delta_i\bigr\}$. We write~$\ran x$ for the
range of~$x$.  Suppose that $\ran x\subseteq (p,q]$ for some
  non-negative integers $p< q$. Then, as observed in \cite[p.~12]{ah},
  the element $u := x|_{\Gamma_q}\in\ell_\infty(\Gamma_q)$ satisfies
\begin{equation}\label{AHp12iq}
  x = i_q(u)\qquad \text{and}\qquad \supp u\subseteq
  \Gamma_q\setminus\Gamma_p.
\end{equation}
Suppose that $x\ne0$, and set $m = \max\ran x\in\N$. Then we define
the \emph{local support} of~$x$ by \[ \locsupp x :=
\supp(x|_{\Gamma_m}) = \{\gamma\in\Gamma_m : x(\gamma)\ne 0\}. \]
Further, suppose that $x = i_n(w)$ for some $n\in\N$ and
$w\in\ell_\infty(\Gamma_n)$. Then we have $n\ge m$ because
$i_n[\ell_\infty(\Gamma_n)] = \spa\{d_\gamma : \gamma\in\Gamma_n\}$,
and hence
\begin{equation}\label{locsuppeq}
 x(\gamma) = \langle e_\gamma^*,i_n(w)\rangle = \langle
P^*_{(0,n]}e_\gamma^*, w\rangle = w(\gamma)\qquad
  (\gamma\in\Gamma_m), \end{equation}
which proves that $\locsupp x = (\supp w)\cap\Gamma_m$. 

We reserve the term `block sequence' for a block sequence with respect
to the FDD $(M_n)_{\in\N}$, in the following precise sense.  Let~$\I$
be a non-empty (finite or infinite) interval of~$\N$.  A \emph{block
  sequence} indexed by~$\I$ is a sequence $(x_i)_{i\in\I}$
in~$X(\Gamma)\setminus\{0\}$ such that \mbox{$x_i\in\spa\{d_\gamma :
  \gamma\in\Gamma\}$} for each $i\in\I$ and $\max\ran x_{i-1}
<\min\ran x_i$ whenever $i\ne \min\I$. 

\subsection*{The set $\mathbf{\Gamma^\textbf{AH}}$} 
Argyros and Haydon's Banach space~$X_{\text{\normalfont{AH}}}$ is the
Bourgain--Del\-baen space $X(\Gamma^{\text{\normalfont{AH}}})$
determined by a very clever choice of $\Gamma^{\text{\normalfont{AH}}}
:= \bigcup_{j\in\N}\Delta_j^{\text{\normalfont{AH}}}$ that we shall
now attempt to describe, following \cite[Section~4]{ah}. The first
step is to fix two fast-increasing sequences $(m_j)_{j\in\N}$ and
$(n_j)_{j\in\N}$ of natural numbers which satisfy the following
conditions (see \cite[Assumption~2.3]{ah}):
\begin{itemize}
\item $m_1\ge 4$ and $n_1\ge m_1^2;$
\item $m_{j+1}\ge m_j^2$ and $n_{j+1}\ge m_{j+1}^2(4n_j)^{\log_2
    m_{j+1}}$ for each $j\in\N$.
\end{itemize}

The recursive definition of the sets
$(\Delta_n^{\text{\normalfont{AH}}})_{n\in\N}$ and the associated
functionals $(c_\gamma^*)_{\gamma\in\Gamma^{\text{\normalfont{AH}}}}$
requires that several other objects are defined simultaneously, as
part of the same recursion.  Indeed, we shall also choose a strictly
increasing sequence $(N_n)_{n\in\N_0}$ of integers and construct four
maps called `$\rank$', `$\age$', $\sigma$ and `$\weight$'. Each of
these maps will be defined on the
set~$\Gamma^{\text{\normalfont{AH}}}$. The first three will take their
values in~$\N$, while `$\weight$' maps into the set $\{1/m_j :
j\in\N\}$.  The map $\sigma$ must be injective and satisfy
$\sigma(\gamma)>\rank\gamma$ for each
$\gamma\in\Gamma^{\text{\normalfont{AH}}}$.

As we have already mentioned, the recursion begins with the set
$\Delta_1^{\text{\normalfont{AH}}} = \{1\}$ and the functional $c_1^*
= 0$.  We set $N_0 = 0$ and define \[ \rank\gamma = \age\gamma =
1,\qquad \sigma(\gamma) = 2\qquad\text{and}\qquad \weight\gamma =
\frac{1}{m_1}\qquad (\gamma =
1\in\Delta_1^{\text{\normalfont{AH}}}). \]
 
Now assume recursively that, for some $n\in\N$, we have defined the
sets
$\Delta_1^{\text{\normalfont{AH}}},\ldots,\Delta_n^{\text{\normalfont{AH}}}$
and the functionals $c_\gamma^*$ for
$\gamma\in\Gamma_n^{\text{\normalfont{AH}}}$ (where
$\Gamma_n^{\text{\normalfont{AH}}} :=
\bigcup_{j=1}^n\Delta_j^{\text{\normalfont{AH}}}$ by convention, as
above), as well as the integers $N_0<N_1<\cdots<N_{n-1}$ and the
maps
$\rank,\age,\sigma\colon\Gamma_n^{\text{\normalfont{AH}}}\to\N$ and
$\weight\colon\Gamma_n^{\text{\normalfont{AH}}}\to\{1/m_j : j\in\N\}$,
where $\sigma$ is injective and satisfies $\sigma(\gamma)>\rank\gamma$
for each $\gamma\in\Gamma_n^{\text{\normalfont{AH}}}$.  Choose
$N_n>N_{n-1}$ such that the set
\begin{align*}
B_{p,n} :=
\biggl\{\!\!\sum_{\ \eta\in\Gamma_n^{\text{\normalfont{AH}}}\setminus\Gamma_p^{\text{\normalfont{AH}}}}\!\!a_\eta
e_\eta^* :\ &a_\eta\in\subfield,\ \sum_{\eta}|a_\eta|\le 1\ \text{and
  the}\\[-3ex] &\text{denominator
  of}\ a_\eta\ \text{divides}\ N_n!\ \text{for
  each}\ \eta\in\Gamma_n^{\text{\normalfont{AH}}}\setminus\Gamma_p^{\text{\normalfont{AH}}}\biggr\}
\end{align*}
is a $2^{-n}$-net in the unit ball of
$\ell_1(\Gamma_n^{\text{\normalfont{AH}}}\setminus\Gamma_p^{\text{\normalfont{AH}}})$
for each $p\in\{0,1,\ldots,n-1\}$, where we have introduced 
$\Gamma_0^{\text{\normalfont{AH}}} := \emptyset$ for convenience. (When
talking about `the denominator' of an element~$a_\eta$ of~$\subfield$
in the complex case, we suppose that $a_\eta$ has been written in the
form $a_\eta = (j+k\mathrm{i})/m$ for some $j,k\in\Z$ and $m\in\N$.)  We
admit into~$\Delta_{n+1}^{\text{\normalfont{AH}}}$ elements~$\gamma$
of two types:
\begin{romanenumerate}
\item Elements of \emph{type 1} are triples of the form \[ \gamma =
  \Bigl(n+1,\frac{1}{m_j},b^*\Bigr), \] where $b^*\in B_{0,n}$ and
  \mbox{$j\in\{1,\ldots,n+1\}$}. If~$j$ is even, then we admit
  each~$\gamma$ of this form
  into~$\Delta_{n+1}^{\text{\normalfont{AH}}}$, whereas if~$j$ is odd,
  we admit~$\gamma$ into~$\Delta_{n+1}^{\text{\normalfont{AH}}}$ if
  and only if $b^* = e_\eta^*$, where
  $\eta\in\Gamma_n^{\text{\normalfont{AH}}}$ has weight $1/m_{4i-2}$
  for some~$i\in\N$, and this weight satisfies
  $1/m_{4i-2}<1/n_j^2$. In both cases we define \[ c_\gamma^*=
  \frac{b^*}{m_j},\qquad \rank\gamma = n+1,\qquad \weight\gamma =
  \frac{1}{m_j}\qquad\text{and}\qquad \age\gamma = 1. \]
\item Elements of \emph{type 2} are quadruples of the form \[ \gamma =
  \Bigl(n+1, \xi,\frac{1}{m_j},b^*\Bigr), \] where
  \mbox{$j\in\{1,\ldots,n+1\}$},
  $\xi\in\Delta_p^{\text{\normalfont{AH}}}$ for some
  $p\in\{1,\ldots,n-1\}$, $\weight\xi = 1/m_j$, \mbox{$\age\xi<n_j$}
  and $b^*\in B_{p,n}$. Again, if $j$ is even, then we admit
  each~$\gamma$ of this form
  into~$\Delta_{n+1}^{\text{\normalfont{AH}}}$, whereas if~$j$ is odd,
  we admit~$\gamma$ into~$\Delta_{n+1}^{\text{\normalfont{AH}}}$ if
  and only if $b^* = e_\eta^*$, where
  $\eta\in\Gamma_n^{\text{\normalfont{AH}}}\setminus\Gamma_p^{\text{\normalfont{AH}}}$
  has weight $1/m_{4\sigma(\xi)}$.  In both cases we define
  \[ \mbox{}\qquad c_\gamma^* = e_\xi^* + \frac{b^* - P_{(0,\,p]}^*b^*}{m_j},\quad
    \rank\gamma = n+1,\quad \weight\gamma = \frac{1}{m_j},\quad \age\gamma
    = 1+\age\xi. \] 
\end{romanenumerate}
It remains to extend the definition of~$\sigma$
to~$\Delta_{n+1}^{\text{\normalfont{AH}}}$.  Set $m =
\max\sigma[\Gamma_n^{\text{\normalfont{AH}}}]$. Then $m>n$, and we may
therefore define $\sigma(\gamma)$ for
$\gamma\in\Delta_{n+1}^{\text{\normalfont{AH}}}$ by assigning to it
any value in $\N\cap(m,\infty)$ that we wish, as long as we choose
distinct values for distinct elements
of~$\Delta_{n+1}^{\text{\normalfont{AH}}}$. This completes the
recursive construction and hence the definition of Argyros and
Haydon's Banach space~$X_{\text{\normalfont{AH}}}$.

\begin{remark}\label{basisofXAHisshrinking}
For later reference, we record the following two facts.
\begin{romanenumerate}
\item\label{basisofXAHisshrinking1} As noted in \cite[p.~17]{ah}, the
  basis constant~$M$ of
  $(d_\gamma^*)_{\gamma\in\Gamma^{\text{\normalfont{AH}}}}$ is at
  most~$2$.
\item\label{basisofXAHisshrinking2} The Schauder basis
  $(d_\gamma)_{\gamma\in\Gamma^{\text{\normalfont{AH}}}}$
  of~$X_{\text{\normalfont{AH}}}$ is shrinking, so that
  $(d_\gamma^*)_{\gamma\in\Gamma^{\text{\normalfont{AH}}}}$ forms a
  Schauder basis for the dual space~$X_{\text{\normalfont{AH}}}^*$,
  and therefore
  $X_{\text{\normalfont{AH}}}^*\cong\ell_1(\Gamma^{\text{AH}})$.
  Indeed, the proof of \cite[Proposition~5.12]{ah} shows that the
  FDD~$(M_n)_{n\in\N}$ for~$X_{\text{\normalfont{AH}}}$ is shrinking,
  and hence the conclusion follows from the elementary general fact
  that if a Schauder basis has a finite-dimensional blocking which is
  shrinking, then the basis is itself shrinking.
\end{romanenumerate}
\end{remark}

We are now ready to define the subspace~$Y$
of~$X_{\text{\normalfont{AH}}}$ that will have the properties stated
in Theorem~\ref{thesubspaceY}.
\begin{definition} 
We begin by recursively defining a sequence $(\Delta_n')_{n\ge2}$ of
non-empty, proper subsets of
$(\Delta_n^{\text{\normalfont{AH}}})_{n\ge2}$.

To start the recursion, we choose an element~$\beta_0$
in~$\Delta_2^{\text{\normalfont{AH}}}$ and set $\Delta_2' =
\{\beta_0\}$. This is certainly a non-empty subset
of~$\Delta_2^{\text{\normalfont{AH}}}$. It is also proper
because~$\Delta_2^{\text{\normalfont{AH}}}$ contains at least two
distinct elements, namely $(2,1/m_2,\pm e_1^*)$.

Now let $n\ge2$, and assume recursively that we have defined
non-empty, proper subsets $\Delta_2',\ldots,\Delta_n'$ of
$\Delta_2^{\text{\normalfont{AH}}},\ldots,\Delta_n^{\text{\normalfont{AH}}}$,
respectively. Set $\Gamma_n' = \bigcup_{j=2}^n\Delta_j'$, and define
\begin{equation}\label{defnDeltanplus1prime} 
\Delta_{n+1}' = \{\gamma\in\Delta_{n+1}^{\text{\normalfont{AH}}}
  : c_\gamma^*(\eta)\neq 0\ \text{for some}\ \eta\in\Gamma_n'\}. 
\end{equation}
Then $\Delta_{n+1}'$ is non-empty because it contains the element
$(n+1,1/m_2,e_{\beta_0}^*)$. To see that
$\Delta_{n+1}'$ is a proper subset of~$\Delta_{n+1}^{\text{\normalfont{AH}}}$, choose
$\zeta\in\Delta_2^{\text{\normalfont{AH}}}\setminus\Delta_2'$. Then we have
\[ \gamma :=
\Bigl(n+1,\frac{1}{m_2},e_{\zeta}^*\Bigr)\in\Delta_{n+1}^{\text{\normalfont{AH}}}, \]
and $c_\gamma^*(\eta) = e_{\zeta}^*(\eta)/m_2 = 0$ for each
$\eta\in\Gamma^{\text{\normalfont{AH}}}\setminus\{\zeta\}\supseteq\Gamma_n'$,
so that $\gamma\notin\Delta_{n+1}'$.  This completes the recursion.

Set $\Gamma' = \bigcup_{n=2}^\infty\Delta_n'$, and define~$Y$ to
be the closed subspace of~$X_{\text{\normalfont{AH}}}$ spanned by the
basic sequence~$(d_\gamma)_{\gamma\in\Gamma'}$.
\end{definition}

The definition of~$Y$ shows immediately that~$Y$ is
infinite-dimensional and has
infinite codimension in~$X_{\text{\normalfont{AH}}}$ (because the sets $\Gamma'$
and $\Gamma^{\text{\normalfont{AH}}}\setminus\Gamma'$ are infinite), and
that $(d_\gamma)_{\gamma\in\Gamma'}$ is a Schauder basis for~$Y$.
This basis is shrinking because it is a subsequence of the shrinking
basis~$(d_\gamma)_{\gamma\in\Gamma^{\text{\normalfont{AH}}}}$
for~$X_{\text{\normalfont{AH}}}$. Thus clauses~\romanref{prop4ofY}
and~\romanref{prop1ofY} of Theorem~\ref{thesubspaceY} are
satisfied. To establish the other two clauses, we require some further
observations concerning~$\Gamma'$ and~$Y$.

\begin{lemma}\label{basiclemmaGammaprime} Let
  $\gamma\in\Gamma^{\text{\normalfont{AH}}}$.  Then: 
\begin{romanenumerate}
\item\label{basiclemmaGammaprime1}
  $\gamma\in\Gamma'\setminus\{\beta_0\}$ if and only if
  $c_\gamma^*|_{\Gamma'}\ne 0;$
\item\label{basiclemmaGammaprime2} $\gamma\in\Gamma'$ if and only if
  $d_\gamma^*|_{\Gamma'}\ne 0$.
\end{romanenumerate}
\end{lemma}
\begin{proof} Set $n = (\rank\gamma)-1\in\N_0$, so that 
$\gamma\in\Delta^{\text{\normalfont{AH}}}_{n+1}$.

\romanref{basiclemmaGammaprime1}. This is (almost) immediate from the
definition of~$\Gamma'$. Indeed, if
$\gamma\in\Gamma'\setminus\{\beta_0\}$, then we have $n\ge2$ and
$c_\gamma^*(\eta)\ne 0$ for some $\eta\in\Gamma_n'$
by~\eqref{defnDeltanplus1prime}, so that $c_\gamma^*|_{\Gamma'}\ne 0$.

Conversely, suppose that $c_\gamma^*(\eta)\ne 0$ for some
$\eta\in\Gamma'$. Then $\rank\eta\le n$ because \mbox{$\supp
  c_\gamma^*\subseteq\Gamma_n$}.  Hence $\eta\in\Gamma'_n$, and
therefore $\gamma\in\Delta_{n+1}'$ by~\eqref{defnDeltanplus1prime}. We
cannot have $\gamma =\beta_0$ because $\rank\beta_0 = 2$, so that
$\supp c_{\beta_0}^*\subseteq\Gamma_1 = \{1\}$, which is disjoint
from~$\Gamma'$.

\romanref{basiclemmaGammaprime2}. Recall that $d_\gamma^* = e_\gamma^*
- c_\gamma^*$.

Suppose first that $\gamma\in\Gamma'$. Then
$d_\gamma^*(\gamma) = 1$ because $c_\gamma^*(\gamma) = 0$, and so
$d^*_\gamma|_{\Gamma'}\ne 0$.

Conversely, suppose that $d_\gamma^*(\eta)\ne 0$ for some
$\eta\in\Gamma'$. If $\gamma = \eta$, then
$\gamma\in\Gamma'$. Otherwise $e^*_\gamma(\eta) = 0$, so that
$c_\gamma^*(\eta)\ne0$, and the conclusion follows
from~\romanref{basiclemmaGammaprime1}.
\end{proof}

\begin{lemma}\label{basiclemmaGammaprime3} Let $\gamma\in\Gamma'.$
  Then $d_\gamma|_{\Gamma^{\text{\normalfont{AH}}}\setminus\Gamma'}
  =0$.
\end{lemma}

\begin{proof} 
We shall prove the result
inductively by showing that $d_\gamma(\eta) = 0$ for each $m\in\N$ and
$\eta\in\Delta_m^{\text{\normalfont{AH}}}\setminus\Delta_m'$.
To begin the induction, we observe that this is true 
whenever $m\le\rank\gamma$
 because
$\supp
d_\gamma\subseteq\{\gamma\}\cup(\Gamma^{\text{\normalfont{AH}}}\setminus
\Gamma_{\rank\gamma}^{\text{\normalfont{AH}}})$ (see \cite[p.~12]{ah}).

Now let $m\ge\rank\gamma$ and
$\eta\in\Delta_{m+1}^{\text{\normalfont{AH}}}\setminus\Delta_{m+1}'$,
and assume inductively that $d_\gamma(\xi) = 0$ for each
$\xi\in\Gamma_m^{\text{\normalfont{AH}}}\setminus\Gamma_m'$.  By
Lemma~\ref{basiclemmaGammaprime}\romanref{basiclemmaGammaprime1}, we
have $c_\eta^*|_{\Gamma'} = 0$ and thus
\[ c_\eta^* = \sum_{\xi\in\Gamma_m^{\text{\normalfont{AH}}}\setminus\Gamma_m'} 
c_\eta^*(\xi)e_\xi^*. \] This implies that
\[ d_\gamma(\eta) = \langle d_\gamma, d_\eta^*+c_\eta^*\rangle = 0 +  
\sum_{\xi\in\Gamma_m^{\text{\normalfont{AH}}}\setminus\Gamma_m'}
c_\eta^*(\xi)d_\gamma(\xi) = 0 \] by the induction hypothesis, and
hence the induction continues.
\end{proof}

To state the following two results concisely, we set $\Gamma_0'
= \Gamma_1' = \emptyset$.

\begin{corollary}\label{corBasicpropertiesofY}
Let $y\in Y$. Then:
\begin{romanenumerate}
\item\label{corBasicpropertiesofY1} $\supp y\subseteq\Gamma'$.
\item\label{corBasicpropertiesofY2} Suppose that $\ran y\subseteq
  (p,q]$ for some non-negative integers $p<q$. Then $y =
  i_q(y|_{\Gamma_q^{\text{\normalfont{AH}}}})$ and $\supp
  (y|_{\Gamma_q^{\text{\normalfont{AH}}}})\subseteq\Gamma_q'\setminus\Gamma_p'$.
\end{romanenumerate}
\end{corollary}

\begin{proof} \romanref{corBasicpropertiesofY1}.  
By the definition of~$Y$, it suffices to show that $\supp
d_\gamma\subseteq\Gamma'$ for each $\gamma\in\Gamma'$, that is,
$d_\gamma(\eta) = 0$ for each
$\eta\in\Gamma^{\text{\normalfont{AH}}}\setminus\Gamma'$, which is
true by
Lemma~\ref{basiclemmaGammaprime3}.

\romanref{corBasicpropertiesofY2}. This follows by
combining~\romanref{corBasicpropertiesofY1} with~\eqref{AHp12iq}.
\end{proof}

\begin{corollary}\label{coriqmapsGamma}
Let $p<q$ be natural numbers. Then
\[ i_q[\ell_\infty(\Gamma_q'\setminus\Gamma_p')] = \spa\{d_\gamma :
\gamma\in\Gamma_q'\setminus\Gamma_p'\}. \]
\end{corollary}

\begin{proof} 
Set $F = \spa\{d_\gamma : \gamma\in\Gamma_q'\setminus\Gamma_p'\}$.
Corollary~\ref{corBasicpropertiesofY}\romanref{corBasicpropertiesofY2}
implies that $F\subseteq
i_q[\ell_\infty(\Gamma_q'\setminus\Gamma_p')]$, so that
\[ |\Gamma_q'\setminus\Gamma_p'| = \dim F\le \dim 
i_q[\ell_\infty(\Gamma_q'\setminus\Gamma_p')]\le\dim
\ell_\infty(\Gamma_q'\setminus\Gamma_p') =
|\Gamma_q'\setminus\Gamma_p'|<\infty. \] Hence
$i_q[\ell_\infty(\Gamma_q'\setminus\Gamma_p')]$ has the same finite dimension
as its subspace~$F$, so they are equal.
\end{proof}

\begin{proof}[Proof of 
Theorem~{\normalfont{\ref{thesubspaceY}\romanref{prop2ofY}}}] Since
  $Y$ has a shrinking basis, its dual is separable, so by a result of
  Lewis and Stegall (see \cite[the second corollary of
    Theorem~2]{ls}), it suffices to show that~$Y$ is a
  $\mathscr{L}_\infty$-space. This follows from an argument similar to
  \cite[Proposition~3.2]{ah}. Indeed, $(\spa\{d_\gamma :
  \gamma\in\Gamma_q'\})_{q=2}^\infty$ is an increasing sequence of
  subspaces of~$Y$ whose union is dense in~$Y$, and these subspaces
  are uniformly isomorphic to the finite-dimensional
  $\ell_\infty$-spaces of the corresponding dimensions
  by~\eqref{theembeddingin} and Corollary~\ref{coriqmapsGamma}
  (applied with $p=1$).
\end{proof}
 
Clause~\romanref{prop3ofY} of Theorem~\ref{thesubspaceY} is, not
surprisingly, significantly harder to prove than
clauses \romanref{prop4ofY}--\romanref{prop2ofY}.  We shall follow
closely Argyros and Haydon's proof of \cite[Theorem~7.4]{ah}, which
shows that all bounded operators on~$X_{\text{\normalfont{AH}}}$ have the form
scalar-plus-compact. `Rapidly increasing sequences' play a central
role in this proof; their definition is as follows.

\begin{definition}\label{defnRIS}
A \emph{rapidly increasing sequence} (or \emph{RIS} for short)
in~$X_{\text{AH}}$ is a block sequence $(x_i)_{i\in\I}$ indexed by a
non-empty (finite or infinite) interval~$\I$ of~$\N$ such that there
are a constant $C>0$ and a strictly increasing sequence
$(j_i)_{i\in\I}$ of natural numbers satisfying
\begin{romanenumerate}
\item\label{defnRIS1} $\| x_i\|_\infty\le C$ for each $i\in\I$;
\item\label{defnRIS2} $\max\ran x_{i-1} < j_i$ for each
  $i\in\I\setminus\{\min\I\}$;
\item\label{defnRIS3} $|x_i(\gamma)|\le C/m_k$ for each $i\in\I$ and
  each $\gamma\in\Gamma^{\text{\normalfont{AH}}}$ with $\weight\gamma
  = 1/m_k$ for some $k\in\N\cap[1,j_i)$.
\end{romanenumerate}
If we need to specify the constant~$C$ in this definition, we refer
to a \emph{$C$-RIS.}

We say that a RIS $(x_i)_{i\in\I}$ is \emph{semi-normalized} if
$\inf_{i\in\I}\|x_i\|_\infty>0$. (Note that
condition~\romanref{defnRIS1}, above, ensures that
$\sup_{i\in\I}\|x_i\|_\infty<\infty$.)

Let $W$ be a subset of~$X_{\text{AH}}$. By a \emph{RIS in~$W$,} we
mean a sequence $(x_i)_{i\in\I}$ that is a RIS in the above sense and
satisfies $x_i\in W$ for each $i\in\I$.
\end{definition}

Our first aim is to establish the following variant of
\cite[Proposition~5.11]{ah} for bounded operators defined on the subspace~$Y$
of~$X_{\normalfont{\text{AH}}}$.
\begin{proposition}\label{AHprop511}
Let~$T$ be a bounded operator from~$Y$ into a Banach space. Then the
following three conditions are equivalent:
\begin{alphenumerate}
\item\label{AHprop511a} every RIS $(x_i)_{i\in\N}$ in~$Y$ has a
  subsequence $(x_i')_{i\in\N}$ such that $\|Tx_i'\|\to0$ as
  $i\to\infty;$
\item\label{AHprop511c} every 
  bounded block sequence  $(x_i)_{i\in\N}$ in~$Y$ has a
  subsequence $(x_i')_{i\in\N}$ such that\\ $\|Tx_i'\|\to0$ as
  $i\to\infty;$
\item\label{AHprop511e} the operator~$T$ is compact. 
\end{alphenumerate}
\end{proposition}

As in \cite{ah}, the proof of this result relies heavily on the
following two notions.

\begin{definition}  A block sequence 
$(x_i)_{i\in\N}$ in~$X_{\normalfont{\text{AH}}}\setminus\{0\}$ has:
\begin{itemize}
  \item \emph{bounded local weight} if $\inf\bigl\{\weight\gamma :
    \gamma\in\bigcup_{i\in\N}\locsupp x_i\bigr\} >0;$
  \item \emph{rapidly decreasing local weight} if, for each $i\in\N$
    and \mbox{$\gamma\in\locsupp x_{i+1}$}, we have $\weight\gamma <
    1/m_{q_i}$, where $q_i := \max\ran x_i$.
\end{itemize}
\end{definition}

\begin{proposition}[{\cite[Proposition~5.10]{ah}}]\label{AHprop510} 
  Let $(x_i)_{i\in\N}$ be a bounded block sequence
  in \mbox{$X_{\normalfont{\text{AH}}}\setminus\{0\},$} and suppose that
  $(x_i)_{i\in\N}$ has either bounded local weight or rapidly
  decreasing local weight. Then $(x_i)_{i\in\N}$ is a RIS.
\end{proposition}

\begin{proof}[Proof of Proposition~{\normalfont{\ref{AHprop511}}}]
The implication
\alphref{AHprop511c}$\Rightarrow$\alphref{AHprop511a} is obvious. 

\alphref{AHprop511c}$\Leftrightarrow$\alphref{AHprop511e}.  Each
subsequence of a bounded block sequence is evidently itself a bounded
block sequence. Hence condition~\alphref{AHprop511c} is equivalent to
the formally stronger statement that $\|Tx_i\|\to0$ as $i\to\infty$
for every bounded block sequence $(x_i)_{i\in\N}$ in~$Y$, and this
latter statement is in turn equivalent to
condition~\alphref{AHprop511e} by
Lemma~\ref{FDDprojapprox}\romanref{FDDprojapprox2}, which applies
because the basis $(d_\gamma)_{\gamma\in\Gamma'}$ for~$Y$ is
shrinking.

It remains to prove that
\alphref{AHprop511a}$\Rightarrow$\alphref{AHprop511c}, which we shall
accomplish by adapting the proof of \cite[Proposition~5.11]{ah}. We
begin by observing that since each subsequence of a RIS is a RIS,
condition~\alphref{AHprop511a} is equivalent to the formally stronger
statement that $\|Tx_i\|\to0$ as $i\to\infty$ for every RIS
$(x_i)_{i\in\N}$ in~$Y$.  Suppose that this statement holds true, let
$(x_j)_{j\in\N}$ be a bounded block sequence in~$Y$, and choose
integers $0 = q_0 < q_1<q_2<\cdots$ such that $\ran x_j\subseteq
(q_{j-1},q_j]$ for each $j\in\N$.  Fix $j,k\in\N$, set $u_j =
  x_j|_{\Gamma_{q_j}^{\normalfont{\text{AH}}}}$ and, for each
  $\gamma\in\Gamma_{q_j}^{\normalfont{\text{AH}}}$, define
  \[ v_j^k(\gamma) = \begin{cases}
     u_j(\gamma)\ &\text{if}\ \weight\gamma\ge
     1/m_k\\ 0\ &\text{otherwise} \end{cases} \qquad\text{and}\qquad
  w_j^k(\gamma) = \begin{cases} u_j(\gamma)\ &\text{if}\ \weight\gamma
    < 1/m_k\\ 0\ &\text{otherwise.} \end{cases} \] Then we have $u_j =
  v_j^k + w_j^k$, $\| v_j^k\|_\infty\vee\| w_j^k\|_\infty =
  \|u_j\|_\infty\le \|x_j\|_\infty$ and \[ \supp v_j^k\cup\supp w_j^k
  = \supp u_j\subseteq\Gamma_{q_j}'\setminus\Gamma_{q_{j-1}}' \] by
  Corollary~\ref{corBasicpropertiesofY}\romanref{corBasicpropertiesofY2}.
  Hence $y_j^k := i_{q_j}(v_j^k)$ and $z_j^k := i_{q_j}( w_j^k)$
  satisfy $y_j^k + z_j^k = i_{q_j}(u_j) = x_j$, they both belong to
  $\spa\{d_\gamma : \gamma\in
  \Gamma_{q_j}'\setminus\Gamma_{q_{j-1}}'\}$ by
  Corollary~\ref{coriqmapsGamma}, and their norms are at most
  $2\|x_j\|_\infty$ by~\eqref{theembeddingin} and
  Remark~\ref{basisofXAHisshrinking}\romanref{basisofXAHisshrinking1}. Thus
  $(y_j^k)_{j\in\N}$ and $(z_j^k)_{j\in\N}$ are bounded block
  sequences in~$Y$.  Using~\eqref{locsuppeq}, we obtain
  \[ \locsupp y_j^k\subseteq\supp v_j^k = \Bigl\{\gamma\in\supp u_j :
  \weight\gamma\ge\frac{1}{m_k}\Bigr\}, \] so that $(y_j^k)_{j\in\N}$
  has bounded local weight, and it is therefore a RIS by
  Proposition~\ref{AHprop510}. Hence the assumption implies that $\|
  Ty_j^k\|\to 0$ as $j\to\infty$, so that we can recursively choose 
  integers $1<j_1<j_2<\cdots$ such that $\| Ty_{j_k}^k\|\to 0$ as
  $k\to\infty$. Set $k_1 = 1$ and, recursively, define $k_{p+1} =
  q_{j_{k_p}}$ for $p\in\N$. Then $(z_{j_{k_p}}^{k_p})_{p\in\N}$ is a
  bounded block sequence with rapidly decreasing local weight, so it
  is a RIS by Proposition~\ref{AHprop510}, and hence
  $\|Tz_{j_{k_p}}^{k_p}\|\to 0$ as $p\to\infty$. It now follows that
  $x_p' := x_{j_{k_p}} = y_{j_{k_p}}^{k_p} +
  z_{j_{k_p}}^{k_p}\ (p\in\N)$ is a subsequence of $(x_j)_{j\in\N}$
  such that $\|Tx_p'\|\to 0$ as $p\to\infty$.
\end{proof}

We shall next establish a lemma which generalizes \cite[Lemma~7.2 and
  Proposition~7.3]{ah}.  While we shall require it only for $\Upsilon
= \Gamma'$, we have chosen to state it in greater generality to
highlight that, unlike Proposition~\ref{AHprop511}, it does not depend
on any special properties of the set~$\Gamma'$.  The statement of this
lemma involves three further notions.  First, for a subspace~$W$
of~$X_{\normalfont{\text{AH}}}$, we denote by
$W\cap\subfield^{\Gamma^{\text{\normalfont{AH}}}}$ the set of $w\in W$
such that $w(\gamma)\in\subfield$ for each
$\gamma\in\Gamma^{\text{\normalfont{AH}}}$, where we recall
that~$\subfield$ is the subfield of the scalar field given
by~\eqref{defnsubfieldL}. Second, for natural numbers $p<q$, we write
$P^*_{(p,\,q]}$ for the operator $P^*_{(0,\,q]} - P^*_{(0,\,p]}$ and
denote by~$P_{(p,\,q]}$ its adjoint. Third, we require the following
piece of terminology, which originates from \cite[Definition~6.1]{ah}.

\begin{definition}
Let $C>0$ and $j\in\N$. A $(C,j,0)$-\emph{exact pair} is a pair
$(z,\eta)\in X_{\text{AH}}\times\Gamma^{\text{AH}}$ that satisfies:
\begin{itemize}
\item $|\langle d_\xi^*,z\rangle|\le C/m_j$ for each
  $\xi\in\Gamma^{\text{AH}}$, $\weight\eta = 1/m_j$, $\|z\|_\infty\le
  C$ and $z(\eta) = 0;$
\item $|z(\xi)|\le C/m_{i\wedge j}$ for each $i\in\N\setminus\{j\}$
  and each $\xi\in\Gamma^{\text{AH}}$ with $\weight\xi = 1/m_i$.
\end{itemize}
\end{definition}

\begin{lemma}\label{AHlemma723}
Let~$C>0$, let $W = \clspa\{d_\gamma : \gamma\in\Upsilon\}$ for some
non-empty subset~$\Upsilon$ of~$\Gamma^{\text{\normalfont{AH}}},$ and
let $T\colon W\to X_{\normalfont{\text{AH}}}$ be a bounded operator.
\begin{romanenumerate}
\item\label{AHlemma723i} Let $(x_i)_{i\in\I}$ be a $C$-RIS in~$W$,
  where~$\I$ is a non-empty interval of~$\N$. Then, for each
  $\epsilon>0$, there is a $(C+\epsilon)$-RIS $(y_i)_{i\in\I}$
  in~$W\cap\subfield^{\Gamma^{\text{\normalfont{AH}}}}$ such that
  $\|x_i-y_i\|_\infty\le \epsilon$ for each $i\in\I$.
\item\label{AHlemma723ii} Suppose that $\dist(Tx_i,\K x_i)\to 0$ as
  $i\to\infty$ for every RIS $(x_i)_{i\in\N}$
  in~$W\cap\subfield^{\Gamma^{\text{\normalfont{AH}}}}$. Then $\dist(Tx_i,\K x_i)\to 0$ as
  $i\to\infty$ for every RIS $(x_i)_{i\in\N}$ in~$W$.
\item\label{AHlemma723iii} Let $\delta>0$, and let $(x_i)_{i\in\N}$ be
  a $C$-RIS in~$W\cap\subfield^{\Gamma^{\text{\normalfont{AH}}}}$ such
  that $\dist(Tx_i,\K x_i)>\delta$ for each $i\in\N$. Then, for each
  $j\in\N$ and $p\in\N_0$, there are $z\in\spa\{x_i:i\in\N\}\subseteq
  W,$ $q\in\N\cap(p,\infty)$ and
  $\eta\in\Delta_q^{\text{\normalfont{AH}}}$ such that the following
  five conditions are satisfied:
\begin{enumerate}
\item\label{AHlemma723iii1} $\ran z\subseteq (p,q);$
\item\label{AHlemma723iii2} $(z,\eta)$ is a $(16C,2j,0)$-exact pair;
\item\label{AHlemma723iii3} $\real(Tz)(\eta)>7\delta/16;$
\item\label{AHlemma723iii4}
  $\|(I_{X_{\normalfont{\text{AH}}}}-P_{(p,\,q]})Tz\|_\infty<\delta/m_{2j};$
\item\label{AHlemma723iii5} $\real\langle Tz,
  P^*_{(p,\,q]}e_\eta^*\rangle > 3\delta/8$.
\end{enumerate}
\item\label{AHlemma723iv} For every RIS $(x_i)_{i\in\N}$ in~$W$,
  $\dist(Tx_i,\K x_i)\to0$ as $i\to\infty$.
\end{romanenumerate} 
\end{lemma}

\begin{proof}
Clauses \romanref{AHlemma723i} and~\romanref{AHlemma723ii} are both
proved by standard approximation arguments. We omit the details.

\romanref{AHlemma723iii}.  Since~$(x_i)_{i\in\N}$ is a bounded block
sequence with respect to the shrinking
basis~$(d_\gamma)_{\gamma\in\Upsilon}$ for~$W$, it is weakly null
in~$W$. Being bounded, the operator~$T$ is automatically weakly
continuous, so that $(Tx_i)_{i\in\N}$ is weakly null
in~$X_{\text{AH}}$. Now the remainder of the proof of
\cite[Lemma~7.2]{ah} carries over verbatim. (Note the need for the
real part in conditions~\eqref{AHlemma723iii3}
and~\eqref{AHlemma723iii5}; this is due to the fact that we consider
complex as well as real scalars.)

\romanref{AHlemma723iv}.  Assume towards a contradiction that there is
a RIS $(x_i)_{i\in\N}$ in~$W$ such that $\dist(Tx_i,\K x_i)\not\to0$
as $i\to\infty$. By~\romanref{AHlemma723ii}, we may suppose that
$x_i\in W\cap\subfield^{\Gamma^{\text{\normalfont{AH}}}}$ for each
$i\in\N$. We may now proceed exactly as in the proof of
\cite[Proposition~7.3]{ah} to reach a contradiction,
using~\romanref{AHlemma723iii} instead of \cite[Lemma~7.2]{ah} and
noting that the element \[ z = \frac{1}{n_{2j_0-1}}\sum_{i=1}^{n_{2j_0
    -1}}z_i \] defined in \cite[p.~34]{ah} belongs to~$W$, so that we
may apply the operator~$T$ to it.
\end{proof}

Finally, we can prove clause~\romanref{prop3ofY} of
Theorem~\ref{thesubspaceY}.

\begin{proof}[Proof of Theorem~{\normalfont{\ref{thesubspaceY}%
      \romanref{prop3ofY}}}]
  Lemma~\ref{AHlemma723}\romanref{AHlemma723iv} shows that, for each
  RIS $(x_i)_{i\in\N}$ in~$Y$, there is a scalar sequence
  $(\lambda_i)_{i\in\N}$ such that $\| Tx_i - \lambda_ix_i\|_\infty\to
  0$ as $i\to\infty$. Suppose that $(x_i)_{i\in\N}$ is
  semi-normalized.  Then, arguing as in the proof of
  \cite[Theorem~7.4]{ah}, we deduce that $(\lambda_i)_{i\in\N}$ is
  convergent and that the limit is independent of the choice of
  $(\lambda_i)_{i\in\N}$ and $(x_i)_{i\in\N}$; that is, we have a
  scalar~$\lambda$ such that
  \begin{equation}\label{proofofprop3ofYeq1}
  \|Tx_i - \lambda x_i\|_\infty\le \|Tx_i - \lambda_i x_i\|_\infty +
  |\lambda - \lambda_i|\,\|x_i\|_\infty\to 0\quad\text{as}\quad
  i\to\infty \end{equation} for every semi-normalized RIS
  $(x_i)_{i\in\N}$ in~$Y$.

  We shall now complete the proof by showing that the operator
  $T-\lambda J$ is compact. By
 Proposition~\ref{AHprop511}, we must show that every RIS
  $(x_i)_{i\in\N}$ in~$Y$ has a subsequence $(x_i')_{i\in\N}$ such
  that $\|Tx_i' - \lambda x_i'\|_\infty\to 0$ as $i\to\infty$. If
  $(x_i)_{i\in\N}$ is semi-normalized, then this follows
  from~\eqref{proofofprop3ofYeq1} (and there is no need to pass to a
  subsequence). Otherwise $(x_i)_{i\in\N}$ has a subsequence
  $(x_i')_{i\in\N}$ which is norm-null, in which case the conclusion
  is obvious (because the operator $T-\lambda J$ is bounded).
\end{proof}

\section{The lattice of closed two-sided ideals of~$\mathscr{B}(Z)$:
  the proofs of Theorem~\ref{2sidedideallatticeofBZ} and
  Proposition~\ref{BZnonamenable}}\label{section3}
\noindent
Denote by~$\mathscr{T}_2$ the algebra of upper triangular $(2\times
2)$-matrices over~$\K$. Since every bounded operator on~$Z$ has a
unique matrix representation of the form~\eqref{matrixrepofoperator},
we can define unital algebra homomorphisms by
\begin{equation}\label{defnphi}
  \phi\colon\ \begin{pmatrix}
  \alpha_{1,1}I_{X_{\text{\normalfont{AH}}}} + K_{1,1} & \alpha_{1,2}
  J + K_{1,2}\\ K_{2,1} & \alpha_{2,2}I_Y +
  K_{2,2} \end{pmatrix}\mapsto \begin{pmatrix} \alpha_{1,1} &
    \alpha_{1,2}\\ 0 & \alpha_{2,2}\end{pmatrix},\quad
  \mathscr{B}(Z)\to\mathscr{T}_2, \end{equation} and
\begin{equation}\label{defnpsi}
\psi\colon\ \begin{pmatrix} \alpha_{1,1} & \alpha_{1,2}\\ 0 &
\alpha_{2,2}\end{pmatrix}\mapsto
\begin{pmatrix} \alpha_{1,1}I_{X_{\text{\normalfont{AH}}}} & 
\alpha_{1,2} J\\ 0 & \alpha_{2,2}I_Y \end{pmatrix},\quad
\mathscr{T}_2\to\mathscr{B}(Z). \end{equation} Clearly $\ker\phi =
\mathscr{K}(Z)$, and the composition~$\phi\circ\psi$ is equal to the
identity operator on~$\mathscr{T}_2$, so that we have a split-exact
sequence
\begin{equation*}
  \spreaddiagramcolumns{2ex}\xymatrix{\{0\}\rto &
    \mathscr{K}(Z)\rto^-{\displaystyle{\iota}} &
    \mathscr{B}(Z)\rto<0.33ex>^-{\displaystyle{\phi}} &
    \mathscr{T}_2\lto<0.33ex>^-{\displaystyle{\psi}}\rto & \{0\},}
\end{equation*} 
where~$\iota\colon\mathscr{K}(Z)\to\mathscr{B}(Z)$ is the inclusion
map.

\begin{proof}[Proof of
  Theorem~{\normalfont{\ref{2sidedideallatticeofBZ}}}] For each
  two-sided ideal~$\mathscr{I}$ of~$\mathscr{T}_2$, the
  pre-image~$\phi^{-1}[\mathscr{I}]$ under~$\phi$ is a two-sided ideal
  of~$\mathscr{B}(Z)$. The identity $\phi^{-1}[\mathscr{I}] =
  \psi[\mathscr{I}] + \mathscr{K}(Z)$ shows that this ideal is closed
  (as the sum of a finite-dimensional subspace and a closed subspace),
  and the map $\mathscr{I}\mapsto \phi^{-1}[\mathscr{I}]$ is an
  order isomorphism of the lattice of two-sided ideals
  of~$\mathscr{T}_2$ onto the lattice of closed two-sided ideals
  of~$\mathscr{B}(Z)$ that contain~$\mathscr{K}(Z)$.
  Since~$X_{\text{\normalfont{AH}}}$ and~$Y$ both have Schauder bases,
  $\mathscr{K}(Z)$ is the minimum non-zero closed two-sided ideal
  of~$\mathscr{B}(Z)$. Hence the conclusion follows from the standard
  elementary fact that the lattice of two-sided ideals
  of~$\mathscr{T}_2$ is given by
  \begin{equation*}
  \spreaddiagramrows{-2ex}\spreaddiagramcolumns{-2ex}%
  \mbox{}\quad\xymatrix{&
    \mathscr{T}_2\ar@{-}[dl]\ar@{-}[dr]\\ \smashw[r]{\biggl\{ \begin{pmatrix}
        \alpha_{1,1} & \alpha_{1,2}\\ 0 & 0 \end{pmatrix} :
      \alpha_{1,1},\,\alpha_{1,2}\in\mathbb{K}\biggr\}
      =\;}\mathscr{R}_1\ar@{-}[dr] & & \mathscr{C}_2\smashw[l]{\;=
      \biggl\{ \begin{pmatrix} 0 & \alpha_{1,2}\\ 0 &
        \alpha_{2,2} \end{pmatrix} :
      \alpha_{1,2},\,\alpha_{2,2}\in\mathbb{K}\biggr\}}\ar@{-}[dl]\\ &
    \smashw[r]{\operatorname{rad}\mathscr{T}_2 =\;}
    \mathscr{R}_1\cap\mathscr{C}_2\smashw[l]{\;=
      \biggl\{ \begin{pmatrix} 0 & \alpha_{1,2}\\ 0 & 0 \end{pmatrix}
      : \alpha_{1,2}\in\mathbb{K}\biggr\}}\ar@{-}[d]\\ &
    \{0\}\smashw{,}} \end{equation*} where
  $\operatorname{rad}\mathscr{T}_2$ denotes the Jacobson radical
  of~$\mathscr{T}_2$, and the lines denote proper inclusions with the
  larger ideal at the top.
\end{proof}

\begin{proof}[Proof of Proposition~{\normalfont{\ref{BZnonamenable}}}] 
  Endow~$\mathscr{T}_2$ with an algebra norm. (Since~$\mathscr{T}_2$
  is finite-dimensional, all norms on it are equivalent, so it does
  not matter which one we choose.) Then~$\mathscr{T}_2$ is a standard
  example of a non-amenable Banach algebra, for instance because the
  map
  \[ \begin{pmatrix} \alpha_{1,1} & \alpha_{1,2}\\ 0 &
    \alpha_{2,2} \end{pmatrix}\mapsto \begin{pmatrix} 0 &
    \alpha_{1,2}\\ 0 & 0\end{pmatrix},\quad
    \mathscr{T}_2\to\operatorname{rad}\mathscr{T}_2, \] is a bounded
    derivation which is not inner (and its codomain
    $\operatorname{rad}\mathscr{T}_2$ is a dual Banach
    $\mathscr{T}_2$\nobreakdash-bi\-module because it is a finite-dimensional
    two-sided ideal of~$\mathscr{T}_2$). Moreover, the map
    \begin{equation}\label{T2isomorphismBK}
  A\mapsto\psi(A)+\mathscr{K}(Z),\quad
  \mathscr{T}_2\to\mathscr{B}(Z)/\mathscr{K}(Z), \end{equation}
    where~$\psi$ is given by~\eqref{defnpsi}, is an algebra
    isomorphism, which is automatically bounded because its domain is
    finite-dimensional, so that~$\mathscr{T}_2$ is isomorphic to a
    quotient of~$\mathscr{B}(Z)$, and hence the conclusion follows
    from \cite[Proposition~2.8.64(ii)]{da}.
\end{proof}

\begin{remark}
  The proof of Proposition~\ref{BZnonamenable} shows that the algebra
  homomorphism~$\phi$ given by~\eqref{defnphi} is bounded because it
  is the composition of the quotient homomorphism of~$\mathscr{B}(Z)$
  onto~$\mathscr{B}(Z)/\mathscr{K}(Z)$ with the inverse of the
  isomorphism~\eqref{T2isomorphismBK}.
\end{remark}

\section{Approximate identities: the proof of
  Theorem~\ref{thmAIs}}\label{sectionAIs}
\noindent
Recall that, for $n\in\N$, $P_{(0,\,n]}|X_{\text{AH}}$ is the $n^{\text{th}}$
canonical projection associated with the shrinking FDD $(M_k)_{k\in\N}
= (\spa\{d_\gamma:\gamma\in\Delta_k^{\text{AH}}\})_{k\in\N}$
for~$X_{\text{AH}}$. Clearly the subspace~$Y$ is
$P_{(0,\,n]}$\nobreakdash-in\-vari\-ant, and the restriction
  $P_{(0,\,n]}|Y$ is the $n^{\text{th}}$ canonical 
projection associated with the shrinking FDD 
$(\spa\{d_\gamma:\gamma\in\Delta_k'\})_{k=2}^\infty$ for~$Y$. 
Consequently Lemma~\ref{FDDprojapprox} implies the following result, 
which establishes all the positive statements concerning the existence
of bounded approximate identities in Theorem~\ref{thmAIs}.
\enlargethispage*{10pt}
\begin{proposition}
  \begin{romanenumerate}
  \item The sequence
    \[ \left(\begin{pmatrix} I_{X_{\text{\normalfont{AH}}}} & 0\\ 0 &
        P_{(0,\,n]}|Y\end{pmatrix}\right)_{n\in\N} \] is a bounded left
    approximate identity in~$\mathscr{M}_1$.
  \item The sequence
    \[ \left(\begin{pmatrix} P_{(0,\,n]}|X_{\text{\normalfont{AH}}} & 0\\ 0 &
        I_Y\end{pmatrix}\right)_{n\in\N} \] is a bounded right
    approximate identity in~$\mathscr{M}_2$.
  \item The sequence
    \[ \left(\begin{pmatrix} P_{(0,\,n]}|X_{\text{\normalfont{AH}}} & 0\\ 0 &
        P_{(0,\,n]}|Y\end{pmatrix}\right)_{n\in\N} \] is a bounded two-sided
    approximate identity in~$\mathscr{K}(Z)$.
  \end{romanenumerate}
\end{proposition}

The non-existence statements in Theorem~\ref{thmAIs} will all be easy
consequences of the following lemma.

\begin{lemma}\label{distJtocompact}
  The inclusion map~$J\colon Y\to X_{\text{\normalfont{AH}}}$ has
  distance~$1$ to the ideal of compact opera\-tors, in the sense that
  \[ \inf\bigl\{ \| J - K\| : K\in\mathscr{K}(Y,
  X_{\text{\normalfont{AH}}})\bigr\} = 1. \]
\end{lemma}
\begin{proof}
  The right-hand side dominates the left-hand side because $\|J\| =
  1$.

  On the other hand, given $\epsilon>0$ and $K\in\mathscr{K}(Y,
  X_{\text{\normalfont{AH}}})$, we can find $n\in\N$ such that $\|K -
  P_{(0,\,n]}K\|\le\epsilon/2$ by
    Lemma~\ref{FDDprojapprox}\romanref{FDDprojapprox1}. Riesz's lemma
    (see, \emph{e.g.}, \cite[Lemma~1.1.1]{cpy}) implies that there
    exists a unit vector $y\in Y$ such that $\| y - P_{(0,\,n]}x\|_\infty\ge
    1-\epsilon/2$ for each $x\in X_{\text{\normalfont{AH}}}$, and hence
  \[ \| J - K\|\ge \| (J - K)y\|_\infty\ge \| y - P_{(0,\,n]}Ky\|_\infty - \| Ky -
  P_{(0,\,n]}Ky\|_\infty\ge 1-\epsilon, \] from which the conclusion
  follows. \end{proof}

\begin{proof}[Proof of Theorem~{\normalfont{\ref{thmAIs}\romanref{thmAIs1}}}]
  For each
  \[ T = \begin{pmatrix} \alpha_{1,1}I_{ X_{\text{\normalfont{AH}}}} +
    K_{1,1} & \alpha_{1,2}J + K_{1,2}\\ K_{2,1} &
    K_{2,2} \end{pmatrix}\in\mathscr{M}_1, \] where
  $\alpha_{1,1},\alpha_{1,2}\in\K$ and $K_{1,1},\ldots,K_{2,2}$ are
  compact, we have
  \[ \left\| \begin{pmatrix} 0 & J\\ 0 & 0\end{pmatrix} -
    \begin{pmatrix} 0 & J\\ 0 & 0\end{pmatrix}T\right\| =
  \biggl\|\begin{pmatrix} -JK_{2,1} & J-JK_{2,2}\\
    0 & 0\end{pmatrix} \biggr\|\ge \| J - JK_{2,2}\|\ge 1 \] by
  Lemma~\ref{distJtocompact}. Hence~$\mathscr{M}_1$ has no right
  approximate identity. 

  The other statements are proved similarly.
\end{proof}

\section{Maximal left ideals of~$\mathscr{B}(Z)$: the proof of
  Theorem~\ref{Thmnonfixedfgmaxleftideal}}\label{sectionMaxLeftIdeals}
\noindent
The key ingredient in our proof of
Theorem~\ref{Thmnonfixedfgmaxleftideal}, besides the properties
of~$X_{\text{\normalfont{AH}}}$ and~$Y$ stated in Theorems~\ref{thmAH}
and~\ref{thesubspaceY}, is the following extension theorem of
Grothendieck (see~\cite[pp.~559--560]{groth}, or
\cite[Theorem~1]{linden}), which applies to compact operators
into~$X_{\text{\normalfont{AH}}}$ or~$Y$ be\-cause they are both
isomorphic preduals of~$\ell_1$.
\begin{theorem}[Grothendieck]\label{lindenstraussextthm}
  Let $E$ be a subspace of a Banach space~$F$, and let $G$ be a Banach
  space whose dual space is isomorphic to~$L_1(\mu)$ for some
  measure~$\mu$. Then every compact operator from~$E$ into~$G$ has an
  extension to a compact operator from~$F$ into~$G$.
\end{theorem}

We shall also require the following elementary observation regarding
the maximal two-sided ideals
\[ \mathscr{R}_1 = \biggl\{ \begin{pmatrix} \alpha_{1,1} & \alpha_{1,2}\\
  0 & 0 \end{pmatrix} :
\alpha_{1,1},\,\alpha_{1,2}\in\mathbb{K}\biggr\}\qquad\text{and}\qquad
\mathscr{C}_2 = \biggl\{ \begin{pmatrix} 0 & \alpha_{1,2}\\ 0 &
  \alpha_{2,2} \end{pmatrix} :
\alpha_{1,2},\,\alpha_{2,2}\in\mathbb{K}\biggr\} \] of~$\mathscr{T}_2$
that were introduced in the proof of
Theorem~\ref{2sidedideallatticeofBZ}. This observation is probably well 
known, but we include a short proof for com\-plete\-ness.

\begin{lemma}\label{maxleftidealsofT2} The 
  ideals~$\mathscr{R}_1$ and~$\mathscr{C}_2$ are the only maximal left
  ideals of~$\mathscr{T}_2$.
\end{lemma}

\begin{proof} Both~$\mathscr{R}_1$ and~$\mathscr{C}_2$ have
  codimension one in~$\mathscr{T}_2$, so that they are maximal as left
  ideals.

  Let~$\mathscr{L}$ be any maximal left ideal of~$\mathscr{T}_2$.  As
  noted in the proof of Theorem~\ref{2sidedideallatticeofBZ}, the
  Jacobson radical of~$\mathscr{T}_2$ is given by
  \[ \operatorname{rad}\mathscr{T}_2 = \biggl\{ \begin{pmatrix} 0 &
    \alpha_{1,2}\\ 0 & 0 \end{pmatrix} :
  \alpha_{1,2}\in\mathbb{K}\biggr\}. \] This ideal is not maximal as a
  left ideal because it is properly contained in~$\mathscr{R}_1$
  and~$\mathscr{C}_2$. Hence the definition of the Jacobson radical as
  the intersection of all the maximal left ideals of~$\mathscr{T}_2$
  implies that~$\mathscr{L}$
  contains~$\operatorname{rad}\mathscr{T}_2$ properly, and
  consequently we can find
  \[ \begin{pmatrix} \alpha_{1,1} & 0\\ 0 &
    \alpha_{2,2} \end{pmatrix}\in\mathscr{L} \] with either
  $\alpha_{1,1}\ne 0$ or $\alpha_{2,2}\ne 0$. In the first case we
  conclude that
  \[ \begin{pmatrix} \beta & 0\\ 0 & 0 \end{pmatrix} = \begin{pmatrix}
    \beta/\alpha_{1,1} & 0\\ 0 & 0 \end{pmatrix}\begin{pmatrix}
    \alpha_{1,1} & 0\\ 0 &
    \alpha_{2,2} \end{pmatrix}\in\mathscr{L}\qquad (\beta\in\K), \] so
  that $\mathscr{R}_1\subseteq\mathscr{L}$, and hence $\mathscr{R}_1 =
  \mathscr{L}$ by the maximality of~$\mathscr{R}_1$. A similar
  argument shows that $\mathscr{L} = \mathscr{C}_2$ in the second
  case.
\end{proof}

\begin{proof}[Proof of Theorem~{\normalfont{\ref{Thmnonfixedfgmaxleftideal}}}]  
  As in the proof of Theorem~\ref{2sidedideallatticeofBZ}, we see that
  \mbox{$\mathscr{L}\mapsto \phi^{-1}[\mathscr{L}]$} defines an order
  isomorphism of the lattice of left ideals of~$\mathscr{T}_2$ onto
  the lattice of closed left ideals of~$\mathscr{B}(Z)$ that
  contain~$\mathscr{K}(Z)$. By \cite[Corollary~4.1]{dkkkl}, every
  non-fixed, maximal left ideal of~$\mathscr{B}(Z)$
  contains~$\mathscr{E}(Z)$ and hence~$\mathscr{K}(Z)$, and therefore
  Lemma~\ref{maxleftidealsofT2} shows that $\mathscr{M}_1 =
  \phi^{-1}[\mathscr{R}_1]$ and $\mathscr{M}_2 =
  \phi^{-1}[\mathscr{C}_2]$ are the only non-fixed, maximal left
  ideals of~$\mathscr{B}(Z)$.

  \romanref{Thmnonfixedfgmaxleftideal1}. For each $T =
  (T_{j,k})_{j,k=1}^2\in\mathscr{M}_1$, the operator $T_{2,2}$ is
  compact, so that it has a compact extension
  $\widetilde{T}_{2,2}\colon X_{\text{\normalfont{AH}}}\to Y$ by
  Theorems~\ref{thesubspaceY}\romanref{prop2ofY}
  and~\ref{lindenstraussextthm}. Moreover, we may express~$T_{1,2}$ in
  the form $T_{1,2} = \alpha_{1,2} J + K_{1,2}$, where
  $\alpha_{1,2}\in\K$ and~$K_{1,2}\colon Y\to X_{\text{\normalfont{AH}}}$ is 
  compact, and then another
  application of Theorem~\ref{lindenstraussextthm} gives a compact
  operator $\widetilde{K}_{1,2}\colon X_{\text{\normalfont{AH}}}\to
  X_{\text{\normalfont{AH}}}$ that extends~$K_{1,2}$. Hence we have
  \[ T = \begin{pmatrix} T_{1,1} & 0\\ T_{2,1} & 0\end{pmatrix} +
  \begin{pmatrix} 0 & T_{1,2}\\ 0 & T_{2,2}\end{pmatrix} =
  T \begin{pmatrix} I_{X_{\text{\normalfont{AH}}}} & 0\\
    0 & 0 \end{pmatrix} + \begin{pmatrix} \alpha_{1,2}
    I_{X_{\text{\normalfont{AH}}}} + \widetilde{K}_{1,2} & 0\\
    \widetilde{T}_{2,2} & 0\end{pmatrix} \begin{pmatrix} 0 & J\\ 0 &
    0 \end{pmatrix}, \] which shows that~$\mathscr{M}_1$ is generated
  as a left ideal by the pair of operators given
  by~\eqref{ThmnonfixedfgmaxleftidealEq2}.

  On the other hand, to see that~$\mathscr{M}_1$ is not generated as a
  left ideal by a single bounded operator, assume the contrary, and let~$R =
  (R_{j,k})_{j,k=1}^2$ be a generator of~$\mathscr{M}_1$. Take
  $\alpha_{1,1},\alpha_{1,2}\in\K$ and compact operators $K_{1,1}$ and
  $K_{1,2}$ such that $R_{1,1} = \alpha_{1,1}
  I_{X_{\text{\normalfont{AH}}}} + K_{1,1}$ and $R_{1,2} =
  \alpha_{1,2} J+K_{1,2}$. Since the operators given
  by~\eqref{ThmnonfixedfgmaxleftidealEq2} both belong
  to~$\mathscr{M}_1$, we can find bounded operators $S = (S_{j,k})_{j,k=1}^2$
  and $T = (T_{j,k})_{j,k=1}^2$ on~$Z$ such that
  \begin{equation}\label{eqNotsinglygenerated} \begin{pmatrix}
      I_{X_{\text{\normalfont{AH}}}} & 0\\ 0 & 0 \end{pmatrix} =
    SR\qquad\text{and}\qquad \begin{pmatrix} 0 & J\\ 0 &
      0 \end{pmatrix} = TR. \end{equation} Write $S_{1,1} = \beta
  I_{X_{\text{\normalfont{AH}}}} + U_{1,1}$ and $T_{1,1} = \gamma
  I_{X_{\text{\normalfont{AH}}}} + V_{1,1}$, where $\beta,\gamma\in\K$
  and $U_{1,1}$ and $V_{1,1}$ are compact. The first part
  of~\eqref{eqNotsinglygenerated} implies that $\beta\alpha_{1,1}=1$
  and $\beta\alpha_{1,2} = 0$, so that necessarily $\alpha_{1,2} = 0$,
  while the second part shows that $\gamma\alpha_{1,1} = 0$ and
  $\gamma\alpha_{1,2} = 1$. This is clearly impossible, and
  hence~$\mathscr{M}_1$ cannot be generated as a left ideal by a
  single operator.

  \romanref{Thmnonfixedfgmaxleftideal2}. Assume towards a
  contradiction that~$\mathscr{M}_2$ is the left ideal
  of~$\mathscr{B}(Z)$ generated by the operators $R_1,\ldots,R_n$ for
  some $n\in\N$. The definition~\eqref{defnmaxideal}
  of~$\mathscr{M}_2$ implies that, for each $j\in\{1,\ldots,n\}$, we
  can find $\beta_j,\gamma_j\in\K$ and $K_j\in\mathscr{K}(Z)$ such
  that $R_j = S_j + K_j$, where
  \begin{equation}\label{defnUj} S_j = \begin{pmatrix} 0 & \beta_j J\\ 
      0 & \gamma_j I_Y \end{pmatrix}. \end{equation} By
  \cite[Corollary~4.7]{dkkkl}, the operator
  \[ \Psi\colon\ z\mapsto (R_1z,\ldots,R_nz),\quad Z\to Z^n, \] is
  bounded below, and it is thus an upper semi-Fredholm operator; that
  is, $\Psi$ has finite-dimensional kernel and closed image. Since the
  set of upper semi-Fredholm operators is closed under compact
  perturbations (see, \emph{e.g.}, \cite[Corollary~1.3.7]{cpy}), the
  operator
  \[ S\colon\ z\mapsto \Psi z - (K_1z,\ldots,K_nz) =
  (S_1z,\ldots,S_nz),\quad Z\to Z^n, \] is also an upper semi-Fredholm
  operator, so that its kernel is finite-dimensional. This, however,
  contradicts the fact that $S(x,0) = (0,\ldots,0)$ for each
  $x\in X_{\text{\normalfont{AH}}}$ by~\eqref{defnUj}.
\end{proof}

\subsection*{Acknowledgements} 
We would like to thank Professor Argyros (National Technical
University of Athens, Greece) for having explained to us how to
construct a subspace~$Y$ which has the properties stated in
Theorem~\ref{thesubspaceY} and Dr Zisimopoulou for helping us overcome
some technical difficulties en route to the proof of
Theorem~\ref{thesubspaceY}\romanref{prop3ofY}.

\end{document}